\documentclass[12pt]{amsart}
\usepackage[T1]{fontenc}
\usepackage[utf8]{inputenc}
\usepackage{lmodern}
\usepackage[english]{babel}
\usepackage[%
    a4paper,
	left=2.5cm,       
	right=2.5cm,      
	top=3.5cm,        
	bottom=3.5cm,     
	heightrounded,    
	bindingoffset=0mm 
]{geometry}
\usepackage{amsmath}
\usepackage{amssymb}
\usepackage{amsfonts}
\usepackage{mathtools}
\usepackage{mathrsfs}
\usepackage{enumerate}
\usepackage{xcolor}
\usepackage[]{hyperref}
\hypersetup{
    colorlinks=true,       % false: boxed links; true: colored links
    linkcolor=blue,          % color of internal links
    citecolor=magenta,        % color of links to bibliography
    filecolor=magenta,      % color of file links
    urlcolor=cyan           % color of external links
}
\usepackage[noabbrev,capitalize,nameinlink]{cleveref}

\usepackage{bookmark}
\usepackage[alphabetic]{amsrefs}
\theoremstyle{plain}
\newtheorem{theorem}{Theorem}[section]
\newtheorem{theoremBis}{Theorem}

\newtheorem{lemma}[theorem]{Lemma}
\newtheorem{prop}[theorem]{Proposition}

\newtheorem{conjecture}[theorem]{Conjecture}
\crefname{conjecture}{Conjecture}{Conjectures}
\newtheorem{corollary}[theorem]{Corollary}
\theoremstyle{definition}

\newtheorem{assumption}[theorem]{Assumption}

\theoremstyle{remark}
\newtheorem{remark}[theorem]{Remark}

\numberwithin{equation}{section}
\newcommand{\eps}{\varepsilon}

\newcommand{\grad}{\nabla}

\newcommand{\norm}[1]{\left\| #1 \right\|}
\newcommand{\abs}[1]{\left| #1 \right|}
\newcommand{\set}[1]{\left\{ #1 \right\}}
\newcommand{\brak}[1]{\left\langle #1 \right\rangle} 
\newcommand{\EE}{\mathbb E}
\newcommand{\dee}{\mathrm{d}}

\title{Non-invariance of Gaussian Measures under the 2D Euler Flow}

\author{Jacob Bedrossian}
\address{
UCLA Department of Mathematics, Los Angeles, CA 90095-1555, USA
}
\email[J. Bedrossian]{jacob@math.ucla.edu}
\author{Micka\"el Latocca}
\address{
Laboratoire de Mathématiques et de Modélisation d'\'Evry (LaMME)\\
Université d'\'Evry\\
23 Bd François Mitterrand, 91000 Évry-Courcouronnes, France
}
\email[M. Latocca]{mickael.latocca@univ-evry.fr}
\usepackage{color}
\usepackage{listings}
\usepackage{setspace}

\definecolor{Code}{rgb}{0,0,0}
\definecolor{Decorators}{rgb}{0.5,0.5,0.5}
\definecolor{Numbers}{rgb}{0.5,0,0}
\definecolor{MatchingBrackets}{rgb}{0.25,0.5,0.5}
\definecolor{Keywords}{rgb}{0,0,1}
\definecolor{self}{rgb}{0,0,0}
\definecolor{Strings}{rgb}{0,0.63,0}
\definecolor{Comments}{rgb}{0,0.63,1}
\definecolor{Backquotes}{rgb}{0,0,0}
\definecolor{Classname}{rgb}{0,0,0}
\definecolor{FunctionName}{rgb}{0,0,0}
\definecolor{Operators}{rgb}{0,0,0}
\definecolor{Background}{rgb}{0.98,0.98,0.98}

\lstdefinelanguage{Python}{
numbers=left,
numberstyle=\footnotesize,
numbersep=1em,
xleftmargin=1em,
framextopmargin=2em,
framexbottommargin=2em,
showspaces=false,
showtabs=false,
showstringspaces=false,
frame=l,
tabsize=4,
basicstyle=\ttfamily\small\setstretch{1},
backgroundcolor=\color{Background},
commentstyle=\color{Comments}\slshape,
stringstyle=\color{Strings},
morecomment=[s][\color{Strings}]{"""}{"""},
morecomment=[s][\color{Strings}]{'''}{'''},
morekeywords={import,from,class,def,for,while,if,is,in,elif,else,not,and,or,print,break,continue,return,True,False,None,access,as,,del,except,exec,finally,global,import,lambda,pass,print,raise,try,assert},
keywordstyle={\color{Keywords}\bfseries},
morekeywords={[2]@invariant,pylab,numpy,np,scipy},
keywordstyle={[2]\color{Decorators}\slshape},
emph={self},
emphstyle={\color{self}\slshape},
}

\begin{document}

%\lstinputlisting[language=Python]{numerics.py}

\begin{abstract}
In this article we consider the two-dimensional incompressible Euler equations and give a sufficient condition on Gaussian measures of jointly independent Fourier coefficients supported on $H^{\sigma}(\mathbb{T}^2)$ ($\sigma>3$) such that these measures are not invariant (in vorticity form).
We show that this condition holds on an open and dense set in suitable topologies (and so is generic in a Baire category sense) and give some explicit examples of Gaussian measures which are not invariant.
We also pose a few related conjectures which we believe to be approachable.  
\end{abstract}

\maketitle

\setcounter{tocdepth}{1}
{\small\tableofcontents}
           
\section{Introduction}

\subsection{The 2D Euler equations and vorticity formulation}

In this article we consider the two-dimensional incompressible Euler equations posed on $\mathbb{T}^2$ in vorticity formulation, written here as
\begin{equation}
    \label{eq.eulerVorticity}
    \tag{E}
    \left\{
    \begin{array}{rcll}
    \partial_t \Omega + u\cdot \nabla \Omega &=& 0 &\text{for } (t,x) \in (0,\infty) \times \mathbb{T}^2 \\ 
    \Omega(0,x) & = &\Omega_0(x) &\text{for } x \in\mathbb{T}^2\,,
    \end{array}
    \right.
\end{equation}
where $\Omega : \mathbb{T}^2 \to \mathbb{R}$ is the scalar vorticity, and $u : \mathbb{T}^2 \to \mathbb{R}^2$ is the velocity which is related to $\Omega$ via the Biot-Savart law, 
\begin{equation}
    \left\{
        \begin{array}{rcll}
            u &=&\grad^\perp (-\Delta)^{-1} \Omega \\ 
         \Omega &=& \nabla \wedge u = \partial_1 u_2 - \partial_2 u_1\,,
        \end{array}
    \right.
\end{equation}
where $\grad^\perp = (-\partial_2, \partial_1)$. 

For $\sigma > 1$, the vorticity equations~\eqref{eq.eulerVorticity} are globally well-posed in $\mathcal{C}^0([0,\infty),H^{\sigma}(\mathbb{T}^2))$ (originally due to Kato \cite{kato1967}, following the older classical work in $C^{0,\alpha}$ H\"older spaces by H\"older and Wolibner \cites{holder,Wolibner33}.
The central question regarding the 2D incompressible Euler equations is to understand the long-time dynamics in various settings.
One of the main questions is to determine whether one has the formation of small scales in the vorticity. 
It is widely expected that `generic' solutions to the 2D Euler equations create smaller scales, leading to unbounded growth in the $H^s$ and $C^{0,\alpha}$ norms. Even more is expected: that generic orbits are not be pre-compact in $L^2$ due to the loss of enstrophy ($L^2$ mass of vorticity) to higher frequencies. The latter is known as \v{S}ver\'{a}k's conjecture (see e.g. \cite{DE22}*{Conjecture 3}). See the recent review article \cite{DE22} and \cref{sec:LitMot} below for more discussion. 

One natural question one can ask about the long-term dynamics of a PDE is to find measures which are left invariant under its flow. 
In this article we are specifically interested in the possible invariance of Gaussian measures under the flow $\Phi_t$ of~\eqref{eq.eulerVorticity} (or quasi-invariance; see \cref{assump.quasinvar} below). 
It is known that the Gibbs measure of~\eqref{eq.eulerVorticity} -- the spatial white noise, supported on vorticities which are $H^{-1-\varepsilon}$ for all $\varepsilon > 0$ -- is invariant under the flow in a suitable sense; see \cites{albeverioCruzeiro, flandoli}. In the non-periodic setting~\cite{CS18}, Gaussian invariant measures have been constructed for the two-dimensional Euler equations at very low regularity (typically $\Omega \in H^{\beta}(\mathbb{R}^2)$, $\beta < -2$).

In this article we show that for sufficiently regular initial data ($\sigma > 3$), Gaussian measures with independent Fourier coefficients are generically not invariant under $\Phi_t$. See \cref{conj1} for related open directions and \cref{sec:LitMot} for more discussion about the motivations and existing literature. 

\subsection{Random initial data and invariant measures} Let $(\Xi, \mathcal{A},\mathbb{P})$ be a probability space.

In this article we solve~\eqref{eq.eulerVorticity} with Gaussian random initial data $\Omega_0$ for~\eqref{eq.eulerVorticity} of the form  
\begin{equation}\label{eq.typicalGaussianData} 
    \Omega_0^{\omega}(x) = \sum_{n \in \mathbb{Z}^2} a_ng_n^{\omega}e^{i n\cdot x }\,,
\end{equation}
where the $(g_n)_{n\in\mathbb{Z}^2}$ are independent and identically distributed standard complex Gaussian random variables, and $(a_n)_{n \in \mathbb{Z}^2}$ is  a complex-valued sequence. Note that replacing $a_n \mapsto a_n e^{i \theta_n}$ for any sequence of phases $\theta_n \in \mathbb R$ defines the same random initial data in law. For this reason we restrict to real-valued $a_n$ without loss of generality. 

Let us be more precise and let $\mathbb{Z}^2_+ \coloneqq \{(n_1,n_2) \in \mathbb{Z}^2, n_1 >0 \text{ or } n_1=0, n_2 \geqslant 0\}$. We will work under the following assumption for the $(g_n)_{n\in\mathbb{Z}^2}$.
\begin{assumption}\label{assump.gn}
    \begin{enumerate}
        \item For all $n \in \mathbb{Z}^2_+$, $g_n \coloneqq r_n+is_n$ where $\{r_n, s_n\}_{n\in \mathbb{Z}^2_+}$ are independent and identically distributed real standard Gaussians $\mathcal{N}(0,1)$. 
        \item $g_{-n}=\bar{g}_n$ for all $n \in \mathbb{Z}^2_+$.
    \end{enumerate}
\end{assumption}

Let us use the norm defined by  
\[
    \norm{(a_n)}_{h^\sigma} \coloneqq \left(\sum_{n \in \mathbb{Z}^2} \brak{n}^{2\sigma} \abs{a_n}^{2} \right)^{1/2},
\] 
and make the following assumptions on the $(a_n)_{n\in\mathbb{Z}^2}$, which are tailored to ensure that almost-surely $\Omega_0^{\omega} \in H^{\sigma}$:
\begin{assumption}\label{assump.an} We assume that $\displaystyle (a_n)_{n\in\mathbb{Z}^2} \in h^\sigma$, where 
\[ 
    h^{\sigma} = \set{(a_n)_{n\in \mathbb{Z}^2}: a_n \in \mathbb{R},\, a_0 = 0, \; a_{-n} = a_n, \; \norm{(a_n)}_{h^\sigma} < \infty}. 
\]
\end{assumption}

We recall that $h^\sigma$ is a Hilbert space equipped with the inner product associated with the complete norm $\norm{\cdot}_{h^\sigma}$. 

Given the isotropic nature of~\eqref{eq.eulerVorticity} we also consider the \textit{radially symmetric} sequences belonging to the space
\begin{align*}
h_{\operatorname{rad}}^\sigma = \set{(a_n) \in h^\sigma : a_m = a_{n}, \text{ for all } m,n \; \textup{such that} \; \abs{m} = \abs{n}}. 
\end{align*}

\begin{remark} \cref{assump.an} and \cref{assump.gn} ensure that~\eqref{eq.typicalGaussianData} is real-valued. The condition $a_0=0$ ensures that $\int_{\mathbb{T}^2}\Omega_0(x) \,\mathrm{d}x =0$, which should be the case since $\Omega_0=\nabla \wedge u(0,\cdot)$. 
\end{remark}

Let us denote by $\mu_{(a_n)}$ the law that the random variable~\eqref{eq.typicalGaussianData} induces on $L^2(\mathbb{T}^2,\mathbb{R})$  
(where we assume that $(g_n)_{n\in\mathbb{Z}^2}$ satisfies \cref{assump.gn}). The set of such corresponding measures is therefore defined by
\begin{equation}
    \label{def.Msig}
    \mathcal{M}^{\sigma} \coloneqq \{ \mu_{(a_n)}, \quad (a_n)_{n\in\mathbb{Z}^2}\in h^{\sigma}\}, 
\end{equation}
as well as its radial counterpart $\mathcal{M}^{\sigma}_{\operatorname{rad}}\coloneqq \{ \mu_{(a_n)}, \; (a_n)_{n\in\mathbb{Z}^2}\in h_{\operatorname{rad}}^{\sigma}\}$.

We remark that for any $\mu\in\mathcal{M}^{\sigma}$ there holds
\begin{equation}
    \label{eq.HsigSupport}
    \mathbb{E}_{\mu}[\|\Omega\|_{H^{\sigma}}^2]\coloneqq \int_{L^2}\,\|\Omega\|_{H^{\sigma}}^2\mathrm{d}\mu(\Omega)=\mathbb{E}[\|\Omega_0^{\omega}\|^2_{H^{\sigma}}] =\sum_{n\in\mathbb{Z}^2} \langle n\rangle^{2\sigma} |a_n|^2 < \infty, 
\end{equation}
which implies that $\operatorname{supp} \mu \subset H^{\sigma}(\mathbb{T}^2)$.

Let us recall that a measure $\mu$  on $H^{\sigma}$ is said to be invariant under $\Phi_t$ if for any $t\in\mathbb{R}$ there holds $(\Phi_t)_*\mu=\mu$, \textit{i.e.\ }for all measurable $A \subset H^{\sigma}$,
\[
    \mu(\Omega_0: \; \Omega(t) \in A)=\mu(\Omega_0: \; \Omega_0 \in A).    
\]
A measure $\mu$ is said to be quasi-invariant under the flow $\Phi_t$ if for all $t\in\mathbb{R}$, $\mu_t \coloneqq (\Phi_t)_*\mu$ is equivalent to $\mu$, that is $\mu_t \ll \mu \ll \mu_t$. In the following we will consider the following quantified version of quasi-invariance: 

\begin{assumption}\label{assump.quasinvar}
There exists $\eta >2$ and  $\alpha >0$ such that for all measurable set $A \subset H^{\sigma}$, there holds for any $t \in \mathbb{R}$, 
\begin{equation}
    \label{eq.quasinvarDef}
    \mu_t(A) \leqslant
        \mu(A)^{(1+|t|^{\eta})^{-\alpha}}.  
\end{equation}
\end{assumption}

\begin{remark} Observe that when $t \to 0$ there holds $\mu(A)^{(1+|t|^{\eta})^{-\alpha}} \to \mu(A)$, and also that $\mu(A)^{(1+|t|^{\eta})^{-\alpha}} \to 1$ as $t\to \infty$, which is consistent with what is expected by the evolution of a quasi-invariant measure, namely that the `information' will deteriorate in large time. 
\end{remark}

\begin{remark}\label{rem.timeReverse} 
Applying~\eqref{eq.quasinvarDef} at time $-t$ instead of $t$ and to the set $\Phi_{-t}A$ instead of $A$ yields $\mu(A) \leqslant \mu_t(A)^{(1+|t|^{\eta})^{-\alpha}}$. Therefore, \cref{assump.quasinvar} should be viewed as a quantitative version of $\mu_t \ll \mu \ll \mu_t$.
For non-time-reversible flows, one would need to explicitly state the reverse of \eqref{eq.quasinvarDef}. 
\end{remark}

\subsection{Main results}
Our first main result is a criterion which detects measures in $\mathcal{M}^\sigma$ which are non (quasi)-invariant under the flow of~\eqref{eq.eulerVorticity}. 

For any sequence $(a_n)$ we introduce the number $\gamma_{s,(a_n)}$ defined by 
\begin{equation}
    \label{eq.gamma}
    \gamma_{s,(a_n)} = \frac{1}{(2\pi)^4}\sum_{n,q \in\mathbb{Z}^2} |a_n|^2|a_q|^2 \frac{(q\cdot n^{\perp})^2}{2}\left(\frac{1}{|n|^2} - \frac{1}{|q|^2}\right) \beta_{n,q},
\end{equation}
where
\begin{align}\label{eq.beta}
\beta_{n,q}=\langle n \rangle ^{2s}\left(\frac{1}{|q|^2}-\frac{1}{|q+n|^2}\right) + \langle q \rangle ^{2s} \left(\frac{1}{|q+n|^2}-\frac{1}{|n|^2}\right) + \langle n + q \rangle^{2s} \left(\frac{1}{|n|^2}-\frac{1}{|q|^2}\right). 
\end{align}
We prove the following result concerning $\gamma_{s,(a_n)}$. 

\begin{theoremBis}\label{thm.main}
Let $\sigma > 3$ and $\mu= \mu_{(a_n)} \in \mathcal{M}^{\sigma}$, where $\mathcal{M}^{\sigma}$ is defined by \eqref{def.Msig}. Let $\gamma_s=\gamma_{s,(a_n)}$, defined by~\eqref{eq.gamma}. If $\gamma_s \neq 0$ for some $s\in(0,\sigma -3)$, then $\mu$ is not quasi-invariant in the sense of \cref{assump.quasinvar} under the flow of~\eqref{eq.eulerVorticity}. In particular $\mu$ is not invariant under the flow of~\eqref{eq.eulerVorticity}. 
\end{theoremBis}

The proof basically proceeds by showing that initial conditions sampled from any invariant measure $\mu$ which satisfies $\gamma\neq 0$ will satisfy
\begin{align*}
\EE[ \norm{\Omega^{\omega}(t)}_{H^s}^2] - \EE[\norm{\Omega^{\omega}_0}_{H^s}^2] \approx \gamma t^2 + \mathcal{O}(t^3) \quad \textup{ as } t \to 0, 
\end{align*}
leading to a contradiction. 
This is done by justifying a Taylor expansion in $t$ for the evolution of the vorticity along the flow of~\eqref{eq.eulerVorticity}. 
It is important to note that the proof uses the assumption of (quasi)-invariance of the measure and the \textit{a priori} estimates it implies in order to justify the approximation.

Our next task is to study the condition $\gamma = 0$ and prove that for a large class of $(a_n)_{n\in \mathbb{Z}^2}$ belonging to $h^{\sigma}$, there holds $\gamma_{s, (a_n)} \neq 0$ for some $s \in (0,\sigma]$. 
To this end, let us introduce the set 
\[
    A^{\sigma} \coloneqq \set{(a_n)_{n\in \mathbb{Z}^2} \in h^{\sigma}, \text{ that there exists } s\in (0,\sigma-3), \gamma_{s,(a_n) 0} \neq 0}.  
\] 
Denote by $\mathcal{A}^{\sigma}\subset \mathcal{M}^{\sigma}$ the set of corresponding measures. 
Similarly we also define $A_{\operatorname{rad}}^{\sigma}$ and $\mathcal{A}_{\operatorname{rad}}^{\sigma}$. \cref{thm.main} guarantees that all of the measures in $\mathcal{A}^{\sigma}$ are not (quasi)-invariant under the two-dimensional Euler flow. Furthermore, we have the following:

\begin{theoremBis}[Genericity in the Baire sense]\label{thm.genericity} Let $\sigma >3$. 
\begin{enumerate}
\item[(i)] The set $A^\sigma$ (resp.\ $A^\sigma_{\operatorname{rad}}$) is an open, dense set in $h^\sigma$ (resp.\ $h^\sigma_{\operatorname{rad}}$).  
\item[(ii)] Similarly, $\mathcal{A}^{\sigma}$ (resp.\ $\mathcal{A}^{\sigma}_{\operatorname{rad}}$) is an open dense set of $\mathcal{M}^{\sigma}$ (resp.\ $\mathcal{M}^{\sigma}_{\operatorname{rad}})$ equipped with the topology induced by the Wasserstein $W_1$ distance associated with the $H^\sigma$ metric.
\end{enumerate}    
\end{theoremBis}

We recall the  Kantorovich-Rubinstein duality formulation of the Wasserstein distance: if $\varphi:H^\sigma \to \mathbb R$, we define 
\[
    \norm{\varphi}_{\text{Lip}} \coloneqq \sup_{f \in H^\sigma} \abs{\varphi(f)} + \sup_{f,g \in H^\sigma} \frac{\abs{\varphi(f) - \varphi(g)} } {\norm{f - g}_{H^\sigma}}, 
\]
\begin{equation}
    \label{eq.KantorovichRubinstein}
    W_1(\mu,\nu)  = \sup_{\substack{\varphi : H^\sigma \to \mathbb R \\ \norm{\varphi}_{Lip} \leqslant 1}} \left\vert\int_{H^\sigma} \varphi(u) \dee \nu(u) - \int_{H^\sigma} \varphi(u) \dee \mu(u)\right\vert.
\end{equation}

\begin{remark}
It follows that the set of all measures in $\mathcal{M}^{\sigma}$ which are invariant under the two-dimensional Euler flow are included in a closed, nowhere dense set of $\mathcal{M}^{\sigma}$ (in the Wasserstein-1 topology), which is hence a meager set in the sense of Baire category.
\end{remark}

We point out that this notion of genericity is strong enough to ensure that many mutually singular measures are included. This will follow from the Feldman-H\'ajek theorem.

\begin{corollary}[Mutually singular non-invariant measures]\label{thm.probaGenericity}
Let $\sigma >3$. Every open set $U \subset h^\sigma$ contains an uncountable set $I \subset U$ such that the corresponding measures $(\mu_x)_{x\in I} \subset \mathcal{M}^{\sigma}$ are not invariant (or quasi-invariant in the sense of \cref{assump.quasinvar}) under the two-dimensional Euler equation and are pairwise mutually singular.
\end{corollary}

One obvious disadvantage of genericity results is that they provide no specific examples.
However, we are able to provide some explicit examples of non-invariant (and invariant) measures.
\begin{theoremBis}[Examples of non-invariant measures]\label{thm.examples}\phantom{}
\begin{enumerate}
    \item[(i)] Let $\mu=\mu_{(a_n)}$, where the sequence $(a_n)_{n\in\mathbb{Z}^2}$ has compact support and satisfies \cref{assump.an}. Then $\mu$ is invariant under the flow of~\eqref{eq.eulerVorticity} if and only if $\operatorname{supp}(a_n)$ is included in a line of $\mathbb{Z}^2$ passing through the origin, or included in a circle centered at the origin. In the case where $\mu$ is not invariant, it is also non-quasi-invariant in the sense of \cref{assump.quasinvar}. \label{thm.exaplesI}
    \item[(ii)] Let $(a_n)_{n\in\mathbb{Z}^2}$ defined by $a_0=0$ and $a_n = \frac{1}{\langle n\rangle^5 \log (3+\langle n\rangle ^2)}$ for $n\neq 0$. Then the measure $\mu = \mu_{(a_n)} \in \mathcal{M}_{\operatorname{rad}}^4$ satisfies $\gamma_{s,(a_n)} > 0$ for some $s \in (0,1)$ and is therefore non-quasi-invariant (in the sense of \cref{assump.quasinvar}) under the flow of~\eqref{eq.eulerVorticity}. In particular $\mu$ is not invariant  under the flow of~\eqref{eq.eulerVorticity}\ref{assump.quasinvar}). \label{thm.exaplesII}
\end{enumerate}
\end{theoremBis}

\begin{remark}
It was proven in \cite{elgindiHuSverak} that the only stationary solutions of~\eqref{eq.eulerVorticity} with compactly supported Fourier transforms are the ones whose Fourier transform support is contained in a line in $\mathbb{Z}^2$ or in a circle of $\mathbb{Z}^2$ centered at the origin. \cref{thm.examples}, Part (i) can be considered an extension of this result to invariant measures. 
\end{remark}

\begin{remark}
The results we prove suggests a difference with the nonlinear dispersive equation setting, for which many quasi-invariance results are known, see e.g. \cites{forlanoSeong22,planchonTzvetkovVisciglia,tzvetkov,OhSosoeTzvetkov18}. In our setting, we conjecture stronger non-quasi-invariance results below (\cref{conj1} and \cref{conj2}). 
To the best of the authors knowledge, non quasi-invariance of (sufficiently regular) Gaussian measures has only been proven for one example of non-dispersive ODE, namely $\partial_t u = |u|^2u$ posed on the torus, which is explicitly solvable, see~\cite{OhSosoeTzvetkov18}*{Theorem 1.6}. It would be interesting to further investigate the possible non-quasi-invariance of Gaussian measures. It seems however that the method used in~\cite{OhSosoeTzvetkov18} is difficult to implement for the Euler equations.
\end{remark}

We end with a few conjectures.
The first conjecture is that Part (i) of \cref{thm.examples} already contains basically all of the counter-examples to non-invariance. 
Notice that \cref{conj1} is specifically relating to Gaussian measures for which the Fourier modes are mutually independent. It is possible that one could construct a variety of invariant Gaussian measures which are of a very different nature (\textit{i.e.\ }supported on quasi-periodic motions). These should also be `non-generic' in some sense, but it is less clear how to characterize this at the current moment.

\begin{conjecture} \label{conj1} Let $\sigma > 0$ (or possibly even $\sigma > -1$ if a suitable local well-posedness theory can be established). A measure $\mu = \mu_{(a_n)} \in \mathcal{M}^{\sigma}$ is quasi-invariant in the sense of \cref{assump.quasinvar} under the flow of~\eqref{eq.eulerVorticity} if and only the sequence $(a_n)_{n \in \mathbb Z^2}$ is supported in a line of $\mathbb{Z}^2$ passing through the origin or on a circle centered at the origin. 
\end{conjecture}
\begin{remark}
Proving non-invariance (rather than non-quasi-invariance) might be a much easier task. 
\end{remark}
\begin{remark}
Note that for $\sigma \in (0,1)$, the random initial data \eqref{eq.typicalGaussianData} is almost surely $W^{\sigma,p}$ for all $p < \infty$, and in particular, the initial conditions are almost surely $C^{0,\alpha}$ for some $\alpha> 0$ (see e.g. \cite{Kahane}). 
Hence, these initial data lie in a classical local well-posedness space. For $\sigma=0$ the initial conditions are almost-surely in $L^2 \cap BMO \setminus L^\infty$ \cite{Kahane}*{Chapter 5} however, we are unsure if it lies in a local well-posedness class (though, we believe it likely does; see \cite{Vishik99} for well-posedness results in borderline regularity spaces). For $\sigma < 0$, initial conditions are almost-surely distribution valued and we are not aware of any theory that would even provide local existence of solutions. 
\end{remark}

\begin{remark} In the recent work~\cite{dS22}, the author constructs a large set of smooth initial data (random Gaussian initial data of the form~\eqref{eq.typicalGaussianData}), which give rise to global solutions to the Euler equations posed on the torus of size $L$. It is proven that in the large box limit $L \gg 1$, the Wick formula remains approximately true. The setting considered in~\cite{dS22} differs from ours, and is closely connected to a large-box effect. Note also that it could be the case that the evolution of a Gaussian measure under the Euler flow remains Gaussian, but that this evolution is not \textit{quasi-invariant}, therefore not contradicting \cref{conj1}. 
\end{remark}

As our article will make it clear, the proof of non-invariance does not require $\gamma_{s,(a_n)} > 0$, but only that $\gamma_{s,(a_n)} \neq 0$.
The proof suggests (but does not quite prove) that the $H^s$ norm of the randomly initialized data grow on average for short time if $\gamma_{s,(a_n)} > 0$, but if $\gamma_{s,(a_n)} < 0$ it is actually decaying.
We are not aware of any example of sequence $(a_n)_{n\in\mathbb{Z}^2}$ such that $\gamma_{s,(a_n)} < 0$.
Specifically, we conjecture that such examples do not exist in $\mathcal{M}^\sigma$ and that data randomly initialized in this manner will grow in $H^s$ on average.
We are unsure under what conditions one could conjecture that short-time norm growth will happen almost-surely.  

\begin{conjecture} \label{conj2}
    Let $\sigma > 0$. If $(a_n)_{ n \in \mathbb Z^2} \in h^{\sigma}$ is not exclusively supported in a line passing through the origin or a circle centered at the origin, then for initial data sampled using \eqref{eq.typicalGaussianData} there exists $C(s)>0$ such that there holds
    \begin{align*}
    \EE \norm{\Omega^\omega (t)}_{H^s} \geqslant \EE \norm{\Omega_0^\omega}_{H^s} + Ct^2. 
    \end{align*}
    for all $t \in [0,t_0)$ and all $s\in(0,\sigma]$, where $t_0=t_0((a_n),s)$.
\end{conjecture}
Of course, we also believe the weaker version of \v{S}ver\'{a}k's conjecture:
\begin{align*}
\lim_{t \to \infty} \EE \norm{\Omega^\omega(t)}_{H^s} = +\infty, 
\end{align*}
and even the original version: that the random orbits $\set{\Omega^\omega(t)}_{t \in \mathbb R}$ are almost-surely not pre-compact in $L^2$. 
However, these stronger conjectures seem very far out of reach at the current time, whereas \cref{conj1} and \cref{conj2} seem likely to be manageable. 

\subsection{Related literature and motivation for the results and conjectures} \label{sec:LitMot}

\subsubsection{Creation of small scales in 2D Euler}\label{sec.introGrowth}

As mentioned in the introduction, it is classical that~\eqref{eq.eulerVorticity} is well-posed in $\mathcal{C}^0([0,\infty),H^{\sigma}(\mathbb{T}^2))$ as soon as $\sigma>1$. Moreover, in~\cite{yudovich} it was proven that for all $t \geqslant 0$: 
\begin{equation}
    \label{eq.growthYudovich}
    \|\Omega(t)\|_{H^{\sigma}} \leqslant C(\Omega_0)e^{e^{Ct}} 
\end{equation}
(a similar double-exponential bound for $C^{0,\alpha}$ was proved earlier by Wolibner \cite{Wolibner33}). 
Optimality of the bound~\eqref{eq.growthYudovich} in generality remains an open question.
On $\mathbb T^2$, the examples of gradient growth are generally based on the Bahouri-Chemin cross~\cite{bahouriChemin}, which uses a hyperbolic point in the flow to drive rapid growth of the gradient at a single point.  
Exponential growth can be obtained~\cite{zlatos}, and double exponential growth can be maintained on finite time~\cite{denisov}.
If one replaces $\mathbb{T}^2$ with the unit disk of $\mathbb{R}^2$, Kiselev and \v{S}ver\'{a}k showed that the double exponential growth \eqref{eq.growthYudovich} is sharp \cite{kiselevSverak} (see \cite{xu} for the extension to any symmetrical smooth domain). In some sense, these examples are also based on the Bahouri-Chemin cross, but also inspired by the computations of finite-time singularity of the 3D Euler equations in a cylinder \cite{LH14} (which has recently been made mathematically rigorous \cite{ChenHou22}).  

Another class of norm growth results is based on shearing.   
The first example of ``generic'' norm growth was demonstrated by Nadirashvili \cite{Nadirashvili1991}, who demonstrated examples of small perturbations of the Couette flow in a periodic channel and the corresponding Taylor-Couette in an annulus which lead to unbounded norm growth like $\norm{\grad \omega(t)}_{L^\infty} \gtrsim t$. 
The recent results of \cite{DEJ22} greatly expand on this idea to produce many more examples of `wandering' solutions and unbounded gradient growth on open sets of initial data in a variety of settings (again near stable stationary states). 
Using the mechanism of inviscid damping, a stronger characterization of small-scale creation -- the loss of $L^2$ pre-compactness of vorticity -- has been proved in a few examples for open sets of small perturbations of special stationary states in sufficiently smooth topologies.
This was first done for the Couette flow in $\mathbb T \times \mathbb R$ \cite{BM13} followed by results near a point vortex in the plane in \cite{IJ22}, and strictly monotone shear flows in a channel satisfying a suitable spectral condition (no embedded eigenvalues) \cites{IJ20,IJ_MS_20,MZ20}. 

\subsubsection{Invariant measures of Euler}
If a Gaussian measure $\mu$ on $H^{\sigma}$ is known to be invariant for the dynamics of a partial differential equations, for which a standard local Cauchy theory is known, then the global solutions $\Omega^{\omega}(t)$ starting at time zero from $\Omega_0^{\omega}$ (which has law $\mu$) almost-surely satisfy the bound
\begin{equation}
    \label{eq.boundGlobalizationBourgain}
    \|\Omega^{\omega}(t)\|_{H^{\sigma}} \leqslant C_{\omega} \sqrt{\log (t+2)}\,,  
\end{equation}
where the random constant $C_{\omega}$ satisfies $\mathbb{E}[e^{cC_{\omega}^2}]<\infty$. This is known as the Bourgain globalization argument~\cite{Bourgain94}. Such a bound for the growth of the Sobolev norms is in sharp contrast with~\eqref{eq.growthYudovich} and seems at odds with the existing norm growth results discussed in \cref{sec.introGrowth} and the observed dynamics of generic initial conditions on the torus or the sphere (see numerical simulations in e.g. \cite{SK98,MV20,DQM15,tabeling2002two} and the references therein and the discussion in \cref{sec:turb} below). 
Note that if the invariant measure $\mu$ only has weaker moment bounds,  or the measure $\mu$ is only quasi-invariant as in \cref{assump.quasinvar} for example, then an upper bound for the growth of Sobolev norms similar to~\eqref{eq.boundGlobalizationBourgain} is known to hold with $\sqrt{\log(t+2)}$  replaced by a different function of time (usually a polynomial) depending on the moments, the PDE, or the quasi-invariance parameters. 

For all $\sigma \geqslant 1$, in~\cites{kuksin,latocca}, invariant measures of the 2D Euler equations were constructed in $H^{\sigma}$ from suitable limits of regularized problems using compactness methods. See also \cites{foldesSy,sy21} for the SQG model or the septic Schrödinger equations cases.  
On the support of these measures, one can show growth estimates of the type
\[
    \|\Omega(t)\|_{H^{\sigma}} \leqslant C_{\omega} (1+t)^{\alpha}\,,    
\]
but an important question remains: what is the support of such invariant measures?
Very little is currently known (for example, we do not know if all such measures are in fact supported only on stationary solutions).
The identification and characterization of `fluctuation-dissipation measures' (\textit{i.e.\ }limits obtained by adding stochastic noise balanced against viscosity and taking inviscid limits, as in for example \cites{kuksin,foldesSy}) and similar limiting measures is mostly open for nonlinear problems.
See \cites{BCZGH16,BBPS18,BCZ15} for the much simpler problem of passive scalars where the corresponding measures can be classified, sometimes exactly. 

\subsubsection{Random initial data and 2D turbulence} \label{sec:turb}
The \emph{direct enstrophy cascade}, \textit{i.e.\ }the tendency for enstrophy to be transferred from large to small scales, is central to the modern understanding of 2D turbulence. 
Turbulence primarily concerns the inviscid limit of the Navier-Stokes equations (rather than directly on Euler) and the 2D turbulence problem is greatly complicated by the presence of the inverse energy cascade.
See \cites{Fjortoft53,Kraichnan67,Leith1968,Batchelor1969} where predictions about this dual cascade were first made and \cites{MontKraichnan1980,tabeling2002two,BE12,Biskamp2003,BCZPSW20,Dudley2023} for more discussion. 
For empirical evidence of 2D turbulence, see the surveys~\cites{BE12,Kellay2002,CerbusPhD2015}, \cites{Lindborg99,Charney1971} for observations from atmospheric data, and e.g. \cites{Gharib1989,Rutgers2001,Kellay2017,Rivera2014,Paret1999} and the references therein for a few of the many experiments.

We remark that physicists and engineers often consider turbulence in a statistically stationary state, \textit{i.e.\ }where external driving produces a continuous input of enstrophy, either from the boundaries or through body forcing, and the $t \to \infty$ limit is taken before the inviscid limit (see discussions in \cite{BE12}), which is somewhat different from a random initial data problem. 
All classical predictions, such as the analogues of the Kolmogorov 4/5 law (see \cites{Bernard99,Lindborg99,Yakhot99,CP17,XB18,BCZPSW20,Eyink96}) and the Batchelor-Kraichnan power spectrum are most easily interpreted in this regime.
Nevertheless, the random initial data problem is observed to lead to closely related dynamics referred to as ``decaying turbulence'', where one sees a transient enstrophy cascade, often generically with a decaying Batchelor-Kraichnan spectrum (see e.g. \cites{Rivera2003,Cardoso1994,Hansen1998} and the references therein).
The inverse cascade, at least on $\mathbb T^2$ or the disk, is observed to lead to behavior sometimes referred to ``condensation'' or ``vortex crystallization'' in which most of the energy congregates into certain large-scale coherent structures; see for example \cite{GY13} and \cite{schecter1999vortex}.   

In any case, the existing norm growth results discussed above, and especially \v{S}ver\'{a}k's Conjecture, is intimately connected to the ability of 2D Euler to send enstrophy to small scales, as observed in 2D turbulence (Shnirelman's Conjecture \cite{DE22}*{Conjecture 4} is also intimately related to the observed condensation effect). 
The observed turbulent dynamics strongly suggest that the Euler equations initialized from data where the spatial and frequency scales are sufficiently  independent are likely to create smaller scales provided the initial conditions are more regular than at least $H^1$ (where the nonlinear interactions of small scales could formally start to dominate) but more likely the result holds at least all the way down to the regularity of the Batchelor-Kraichnan spectrum (corresponding to $H^{-\eps}$ regularity, \textit{i.e.\ }$a_n = \brak{n}^{-2}$).
It even seems plausible that anything more regular than the Gibbs measure should not be invariant unless it is extremely degenerate.
These ideas motivate both \cref{conj1,conj2}.

\subsection{Plan of the paper}

In \cref{sec.prelim} we gather some standard fact about the bilinear form of the Euler equations for the convenience of the reader.

In \cref{sec.main} we provide the reader with the main idea underlying the proof of \cref{thm.main}, which is then reduced to proving two key statements : \cref{prop.orthogonality}, which is the goal of \cref{sec.ortho} and \cref{prop.estW} which is the goal of \cref{sec.perturbative}. 

Then \cref{sec.genericity} is devoted to proving \cref{thm.genericity}, \cref{thm.probaGenericity} and \cref{thm.examples}. 

\subsection*{Acknowledgments} 

M.\,L. thanks Nikolay~Tzvetkov for interesting discussions, as well as Maxime~Breden for suggesting the interval arithmetic package for python.

The research of J\,.B. was supported by NSF Award DMS-2108633.

\section{Preliminaries}\label{sec.prelim}

\subsection{Notations used}

We write $\langle n \rangle = \sqrt{1 + |n|^2}$, where $|n|^2=n_1^2+n_2^2$ for $n=(n_1,n_2)\in\mathbb{Z}^2$ and $n^{\perp}=(-n_2,n_1)$.

We recall that the Fourier transform of a regular-enough $f : \mathbb{T}^2 \to \mathbb{R}^N$ is the function $\hat f : \mathbb{Z}^2 \to \mathbb{R}^N$ defined by  
\[
    \hat f (n)=\int_{\mathbb{T}^2}f(x)e^{-in\cdot x}\,\mathrm{d}x\,,    
\]
which we simply refer to as $\hat f_n$ for notational simplicity. We also write $\mathcal{F}(u)$ in place $\hat{u}$. The Fourier inversion formula therefore reads  
\[
    f(x)= \frac{1}{(2\pi)^2}\sum_{n\in\mathbb{Z}^2} \hat f_n e^{in\cdot x}.    
\]

We define the space $H^{\sigma}(\mathbb{T}^2)$ as the set of functions $f$ such that $\|f\|_{H^{\sigma}}<\infty$ where 
 \[
    \|f\|_{H^{\sigma}} \coloneqq \left(\sum_{n\in\mathbb{Z}^2} \langle n \rangle ^{2\sigma}|\hat{f}_n|^2\right)^{\frac{1}{2}}.
\]

\subsection{The bilinear form}

Let us denote $U[\Omega]$ the vector field $\mathbb{T}^2 \to \mathbb{R}^2$ defined by $U[\Omega]=\nabla^{\perp}(-\Delta)^{-1}\Omega$
which is uniquely defined in the class of divergence-free vector fields. We recall that thanks to Calder\'on-Zygmund estimates, the linear operator $a \mapsto \nabla U[a]$ is continuous $L^p \to L^p$ for any $p\in (1,\infty)$.

We define the following bilinear form for any smooth enough functions $a, b : \mathbb{T}^2 \to \mathbb{R}$ by 
\begin{equation}
    \label{eq.bilinear} 
    B(a,b)=\frac{1}{2}U[a]\cdot \nabla b + \frac{1}{2}U[b]\cdot \nabla a, 
\end{equation}
and use the notation $B_1(\Omega) \coloneqq B(\Omega, \Omega) = U[\Omega]\cdot \nabla \Omega$. 

We will also use the following notation: 
\begin{equation}
    \label{eq.B3}
    B_2(\Omega)= B(\Omega,B_1(\Omega))\,,
\end{equation}
\begin{equation}
    \label{eq.B4}
    B_3(\Omega)= B(B_1(\Omega),B_1(\Omega)) + 2B(\Omega,B_2(\Omega))\,, 
\end{equation}
\begin{equation}
    \label{eq.B5}
    B'_3(\Omega)= B(B_1(\Omega),B_2(\Omega))\,,     \tilde{B}_3(\Omega)= B(B_2(\Omega),B_2(\Omega))\,.
\end{equation}

\begin{lemma}[Properties of the $B_i$]\label{lemm.bilinearForm} Let $a, b, \Omega : \mathbb{T}^2 \to \mathbb{R}$. Then: 
\begin{enumerate}
    \item[(i)] $(U[b]\cdot \nabla a,a)_{L^2}=0$.
    \item[(ii)] For any $s>0$, and any $0<\varepsilon < s$ there holds  
    \begin{equation}
        \label{eq.Baaa}
        |\langle B(a,a),a \rangle_{H^s}| \leqslant C(s,\varepsilon) \|a\|_{H^{1+\varepsilon}}\|a\|_{H^{s}}^2\,, 
    \end{equation} 
    \begin{equation}
        \label{eq.Baba}
        |\langle B(a,b),a \rangle_{H^s}| \leqslant C(s,\varepsilon) \|b\|_{H^{s+1+\varepsilon}}\|a\|_{H^{s}}^2.
    \end{equation} 
    \item[(iii)] Similarly for any $0<\varepsilon < s$ there holds 
        \begin{equation}
            \|B_1(\Omega)\|_{H^s} \leqslant C(s,\varepsilon) \|\Omega\|^2_{H^{s+1}}\,,
        \end{equation} 
        \begin{equation}
            \|B_2(\Omega)\|_{H^s}\leqslant C(s,\varepsilon) \|\Omega\|^3_{H^{s+2}}\,,
        \end{equation}
        \begin{equation}
            \|B_3(\Omega)\|_{H^s}\leqslant C(s,\varepsilon) \|\Omega\|^4_{H^{s+3}}\,,
        \end{equation}
        \begin{equation}
            \|B'_3(\Omega)\|_{H^s} \leqslant C(s,\varepsilon)\|\Omega\|^5_{H^{s+3}} \;\;\text{ and }\;\; \|\tilde{B}_3(\Omega)\|_{H^s} \leqslant C(s,\varepsilon)\|\Omega\|^6_{H^{s+3}}\,,
        \end{equation}
\end{enumerate}
\end{lemma}
    
\begin{proof} Part (\textit{i}) is a consequence of the fact that $U[b]$ is divergence-free. 
    
\textit{(ii)}: Estimate~\eqref{eq.Baaa} is a consequence of the Kato-Ponce commutator estimate, (see for instance~\cite{katoPonce}), 
\[
    [\langle \nabla \rangle^s,U[a] \cdot \nabla ]= \mathcal{O}_{H^s \to L^2}(\|\nabla U[a]\|_{L^{\infty}})
\] and~\textit{(i)}, as well as the Sobolev embedding $H^{1+\varepsilon} \hookrightarrow L^{\infty}$ and the fact that $a \mapsto \nabla U[a]$ is continuous as a linear map $H^s \to H^s$. Estimate~\eqref{eq.Baba} follows from similar computations and product estimates on $U[a]\cdot \nabla b$ in $H^s$.  

Finally, \textit{(iii)} follows from product laws in Sobolev spaces. 
\end{proof}

\subsection{Invariant measures bounds}

We recall the following almost-sure upper bounds for the growth of Sobolev norms for solutions to~\eqref{eq.eulerVorticity} with initial data sampled on the support of an invariant Gaussian measure. 

\begin{prop}[\cite{Bourgain94}]\label{prop.globalBound}
Let $\sigma>2$ and $\mu$ a Gaussian measure on $H^{\sigma}(\mathbb{T}^2)$ which is either invariant under the flow of~\eqref{eq.eulerVorticity} or satisfies the quantitative quasi-invariance property of \cref{assump.quasinvar}. Then there exists a deterministic constant $c >0$ and a random constant $C_{\omega}$ such that for almost every $\omega$ there holds for any $t\geqslant 0$,  
    \begin{equation}
        \label{eq.invarianceBound}
        \|\Omega^{\omega}(t)\|_{H^{\sigma}} \leqslant  \left\{
            \begin{array}{rl}
                C_{\omega}\sqrt{\log(2+t)} & \text{ in the invariant case,}\\  
                C_{\omega}(1+t^{\eta})^{\frac{\alpha}{2}}\sqrt{\log(2+t)} & \text{ in the quasi-invariant case,}
            \end{array}
        \right.
    \end{equation}
    and where $\mathbb{E}[e^{c C_{\omega}^2}] < \infty$.
\end{prop}

\begin{proof} The proof in the invariant case can be found in~\cite{Bourgain94}, so let us illustrate the argument by proving the very similar case of the quantitative quasi-invariance case. 
    
\textit{Step 1: Large measure, finite time.} Let $T>0$, $\varepsilon >0$ and $\lambda = \lambda (\varepsilon, T) >0$ which will be chosen later. By standard local well-posedness of~\eqref{eq.eulerVorticity} in $H^{\sigma}$, there exists $c>0$ such that if $\tau = \frac{c}{\lambda}$ and $\|\Omega(t_0)\|_{H^{\sigma}} \leqslant \lambda$ then $\|\Omega(t)\|_{H^{\sigma}}\leqslant 2\lambda$ for all $t_0 \leqslant t \leqslant t_0 +\tau$. 
Let $B_{\lambda}\coloneqq \{\Omega_0 \in H^{\sigma}: \|\Omega_0\|_{H^{\sigma}}\leqslant \lambda \}$, so that by the fact that $\mu$ is Gaussian there holds $\mu(H^{\sigma}\setminus B_{\lambda}) \leqslant e^{-c\lambda^2}$. Let us introduce 
\[
    G_{\lambda} \coloneqq \bigcap_{n=0}^{\lfloor T/\tau \rfloor} (\Phi_{n\tau})^{-1}(B_{\lambda}),     
\] 
where $\Phi_t$ is the flow of~\eqref{eq.eulerVorticity}. We have 
\begin{align*}
    \mu(H^{\sigma}\setminus G_{\lambda}) & \leqslant \sum_{n=0}^{\lfloor T/\tau \rfloor} \mu\left(\Phi_{n\tau})^{-1}(H^{\sigma} \setminus B_{\lambda})\right) \\
    & \leqslant \sum_{n=0}^{\lfloor T/\tau \rfloor} \mu\left(H^{\sigma} \setminus B_{\lambda}\right)^{(1+(n\tau)^{\eta})^{-\alpha}} \leqslant T\tau^{-1} e^{-c\lambda^2(1+T^{\eta})^{-\alpha}},    
\end{align*}
where we have used~\eqref{eq.quasinvarDef} in the second inequality. Therefore there exists a numerical constant $C>0$ such that  
\[
    \mu(H^{\sigma}\setminus G_{\lambda}) \leqslant cT\lambda e^{-c\lambda^2(1+T^{\eta})^{-\alpha}} \leqslant CTe^{-\frac{c}{2}\lambda^2(1+T^{\eta})^{-\alpha}}. 
\]
Let us choose $\lambda \sim (1+T^{\eta})^{\frac{\alpha}{2}} \sqrt{\log(T)+\log(\varepsilon^{-1})}$ and let $A_{T,\varepsilon} = \{\omega: \Omega_0^{\omega} \in G_{\lambda}\}$. Then any $\omega \in A_{T,\varepsilon}$ is such that $\|\Omega^{\omega}(t)\|_{H^{\sigma}} \leqslant C (1+T^{\eta})^{\frac{\alpha}{2}} \sqrt{\log(T)+\log(\varepsilon^{-1})}$ 
for all $0\leqslant t \leqslant T$, and $\mathbb{P}(\omega \notin A_{T,\varepsilon}) \leqslant \varepsilon$.  

\textit{Step 2: Global in time almost-sure bounds.} We proceed with a standard Borel-Cantelli argument. Let us consider the probability set 
$A_{\varepsilon} \coloneqq \bigcap_{n\geqslant 1} A_{2^n,2^{-n}\varepsilon}$, which satisfies $\mathbb{P}(\omega \notin A_{\varepsilon}) \leqslant \sum_{n\geqslant 1} \mathbb{P}(\omega\notin A_{2^n,\varepsilon 2^{-n}}) \leqslant \varepsilon$. Let $\omega \in A_{\varepsilon}$ and $t \geqslant 0$. Fix $n$ such that $t \in [2^{n-1},2^{n}]$. Since $\omega \in A_{2^n,\varepsilon 2^{-n}}$, we have 
\[
    \|\Omega^{\omega}(t)\|_{H^{\sigma}} \leqslant C (1+(2^n)^{\eta})^{\frac{\alpha}{2}} \sqrt{\log(2^n)+\log(\varepsilon^{-1}2^n)} \lesssim \sqrt{\log(\varepsilon^{-1})}(1+t^{\eta})^{\frac{\alpha}{2}} \sqrt{\log(2+t)},
\]
which therefore holds for all $t \geqslant 0$ by definition of $n$. 

Finally, the set $A\coloneqq \bigcup_{k\geqslant 1} \bigcap_{j\geqslant k}A_{2^{-j}}$ satisfies $\mathbb{P}(A)=1$ and for any $\omega \in A$, ~\eqref{eq.invarianceBound} holds. Moreover, $\{C_{\omega}>\lambda\} \subset A_{2^{-j}}$ where $j$ is such that $\sqrt{j} \geqslant \lambda$, therefore, $\mathbb{P}\left(C_{\omega}>\lambda\right)\leqslant 2^{-j} \leqslant e^{-c\lambda^2}$, which ensures $\mathbb{E}[e^{-cC_{\omega}^2}] < \infty$. 
\end{proof}

\section{Proof of \cref{thm.main}}\label{sec.main}

The main idea in the proof of \cref{thm.main} is that on short time we expect the following expansion: 
\begin{equation} 
    \label{eq.ansatz}
    \Omega^{\omega}(t) \simeq \Omega^{\omega}_0 - tB_1(\Omega^{\omega}_0) + t^2 B_2(\Omega^{\omega}_0)\,,
\end{equation}
and that this ansatz is responsible for some change of the Sobolev norms. More precisely the crucial claim is the following:

\begin{prop}\label{prop.orthogonality} Let $s > 0$.
Then, there holds 
\[
    \|\Omega_0^{\omega} -t B_1(\Omega_0^{\omega}) +t^2B_2(\Omega_0^{\omega})\|_{L^2_{\omega}H^s}^2 = \|\Omega_0^{\omega}\|_{L^2_{\omega}H^s}^2 + \gamma t^2 + \delta t^4\,,  
\]
where $\delta \coloneqq \|B_3(\Omega_0^{\omega})\|^2_{L^2_{\omega}H^s}$ and 
\[
    \gamma \coloneqq \|B_1(\Omega_0^{\omega})\|^2_{L^2_{\omega}H^s} +2\mathbb{E}\left[\langle\Omega_0^{\omega},B_2(\Omega_0^{\omega})\rangle_{H^s}\right]\,, 
\]
is the constant given by~\eqref{eq.gamma}.
\end{prop} 

The ansatz~\eqref{eq.ansatz} is not an equality and the remainder is given by 
\begin{equation}
    \label{eq.remainderDef} 
    w^{\omega}(t) \coloneqq \Omega^{\omega}(t) - \Omega^{\omega}_0 + tB_1(\Omega_0^{\omega}) - t^2B_2(\Omega^{\omega}_0).
\end{equation}
We next aim to show that this remainder is small for $s \in (0,\sigma-3)$ (the loss of three derivatives is coming from the number of derivatives in the ansatz, \textit{i.e.\ }in $B_2$). 
Given the rapid nonlinear instabilities inherent in the Euler equations it seems difficult to directly justify this ansatz in $H^s$.
However, the assumption of invariance or quasi-invariance implies significant a priori estimates coming from globalization (Proposition \ref{prop.globalBound}). These are leveraged in $H^s$ energy estimates to yield the following stability estimate.  

\begin{prop}\label{prop.estW} Assume that the law of $\Omega^{\omega}_0$ is given by $\mu_{(a_n)}\in \mathcal{M}^{\sigma}$ for some $\sigma >3$, and let $s \in (0,\sigma -3)$.
Suppose that $\mu_{(a_n)}$ is invariant or quasi-invariant in the sense of Assumption \ref{assump.quasinvar}.  
Then there exists a (deterministic) time $t_1\coloneqq t_1(s,\sigma, (a_n))>0$ and a (deterministic) constant $C \coloneqq C(s, \sigma,(a_n))>0$ such that  
    \begin{equation}
        \|w^{\omega}(t)\|_{L^2_{\omega}H^s} \leqslant Ct^3 \text{ for all } t \in [0,t_1]\,, 
    \end{equation}
where $w^{\omega}$ is defined by~\eqref{eq.remainderDef}.
\end{prop}

Once this \textit{a priori} estimate on the remainder is obtained, one can reach a contradiction with invariance using the fact that under the assumption of invariance
\begin{align*}
\norm{\Omega_0^\omega}_{L^2_{\omega}H^s}^2 = \norm{\Omega^\omega(t)}_{L^2_{\omega}H^s}^2. 
\end{align*}

In the quasi-invariant case we will use a similar result: 

\begin{lemma}\label{lem.quasinvarDifference} Let $\sigma >0$ and $\mu\in\mathcal{M}^{\sigma}$ satisfying \cref{assump.quasinvar}. Let $s\in (0,\sigma]$. There exists a time $t_1=t_1(s,(a_n),\alpha,\eta)>0$ and $C=C(s,(a_n),\alpha,\eta)>0$ such that for all $0\leqslant t \leqslant t_1$ there holds 
\begin{equation}
    \label{eq.quasinvarB}
    \left\vert \mathbb{E}_{\mu_t}[\|\Omega\|^2_{H^s}] - \mathbb{E}_{\mu}[\|\Omega\|^2_{H^s}]  \right\vert \leqslant Ct^{\eta}, 
\end{equation}
where we recall that $\eta >2$ is the constant appearing in \cref{assump.quasinvar}.
\end{lemma}
\begin{proof}
Recall that there exists $c_1>0$ such that for all $\lambda >0$, there holds 
\begin{equation}
    \label{eq.upper}
    \mu(\|\Omega\|_{H^s}>\lambda) \leqslant e^{-c_1\lambda^2}. 
\end{equation}
We will use also use two lower bounds: there exists $\lambda _0 >0$ such that for all $\lambda \geqslant \lambda_0$ we have
\begin{equation}
    \label{eq.lower1}
    \mu(\|\Omega\|_{H^s}>\lambda) \geqslant e^{-c_2\lambda^2},
\end{equation}
and there exists $c_3>0$ such that for any $\lambda \in [0,\lambda_0]$ there holds
\begin{equation}
    \label{eq.lower2}
   \mu(\|\Omega\|_{H^s}>\lambda) \geqslant c_3\lambda^2. 
\end{equation}

Let us start with explaining how to obtain~\eqref{eq.lower1}: assume that $a_{(1,0)} \neq 0$ (this can be done without loss of generality, up to considering some $n\in\mathbb{Z}^2$ such that $a_n \neq 0$). Write $g_{(1,0)}^{\omega}=g^{\omega} + i h^{\omega}$ where $g^{\omega}$ and $h^{\omega}$ are independent standard real-valued Gaussian random variables. Then $a_{(1,0)}g^{\omega}_{(1,0)} + a_{(-1,0)}g^{\omega}_{(-1,0)}=2a_{(1,0)}g^{\omega}$ so that we infer that  
\begin{align}
    \label{eq.subset}
       \left\{|g^{\omega}|>\frac{\lambda}{|a_{1,0}|}\right\} \times \left\{\left\|\sum_{n \notin \{(1,0),(-1,0)\}} a_ng_n^{\omega}\right\|_{H^s} \leqslant \lambda\right\} \subset \{\|\Omega\|_{H^s}>\lambda\}.
\end{align}
Using independence we find 
\begin{align*}
    \mu\left(\|\Omega\|_{H^s}>\lambda\right) &\geqslant \mathbb{P}\left(|g^{\omega}|>\frac{\lambda}{|a_{1,0}|}\right)\mathbb{P}\left(\left\|\sum_{n \notin \{(1,0),(0,1)\}} a_ng_n^{\omega}\right\|_{H^s} \leqslant \lambda\right) \\
    & \geqslant  \mathbb{P}\left(|g^{\omega}|>\frac{\lambda}{|a_{1,0}|}\right) (1-e^{-c_1\lambda^2}).
\end{align*}
The conclusion now follows from the fact that for $\lambda$ large enough there holds $(1-e^{-c_1\lambda^2})\geqslant \frac{1}{2}$ for $\lambda$ large enough and the fact that for $\Lambda$ large enough there holds
\begin{equation}
    \label{eq.lower3}
    \mathbb{P}\left(|g^{\omega}|>\Lambda\right) \geqslant e^{-c_2\Lambda^2}.
\end{equation}
This last inequality follows from the fact that: 
\[
    \int_{\Lambda}^{\infty}e^{-u^2}\,\mathrm{d}u = \frac{e^{-\Lambda^2}}{2\Lambda} -
    \int_{\Lambda}^{\infty} \frac{e^{-u^2}}{2u^2}\,\mathrm{d}u \sim \frac{e^{-\Lambda^2}}{2\Lambda} \geqslant e^{-2\Lambda^2},
\]
as $\Lambda \to \infty$, which results from the fact that the remainder integral is negligible with respect to the left-hand side integral. 

To prove~\eqref{eq.lower2}, we use~\eqref{eq.subset} and that $\lambda \leqslant \lambda_0$ to bound
\[
    \mu\left(\|\Omega\|_{H^s}>\lambda\right) \geqslant \mathbb{P}\left(|g_{1,0}^{\omega}|>\frac{\lambda_0}{|a_{1,0}|}\right)(1-e^{-c_1\lambda^2}) \geqslant c_3\lambda^2
\]
for a small enough $c_3=c_3(\lambda_0)>0$ for all $\lambda\in[0,\lambda_0]$.

In order to make the following computations simpler in the proof of the lemma, let us observe that since $(1+|t|^\eta)^{-\alpha}=1-\alpha |t|^{\eta} + \mathcal{O}(|t|^{2\eta})$ as $t \to 0$, one can write 
\begin{equation}
    \label{eq.defK}
    (1+|t|^{\eta})^{-\alpha}=1-\alpha |t|^{\eta}(1+K(t)),
\end{equation}
where $K(t)$ satisfies $\frac{1}{2}\leqslant 1 + K(t) \leqslant 2$ for all $t\in[-t_0,t_0]$, with $t_0=t_0(\alpha,\eta)>0$. We turn to the proof of the lemma and start by writing 
\begin{align*}
     \mathbb{E}_{\mu_t}[\|\Omega\|^2_{H^s}] - \mathbb{E}_{\mu}[\|\Omega\|^2_{H^s}] = 2\int_0^{\infty} \lambda (\mu_t(A_{\lambda})-\mu(A_{\lambda})) \dee \lambda
\end{align*}
where $A_\lambda = \{\Omega\in H^s: \|\Omega\|_{H^s}>\lambda\}$. Next, we use~\eqref{eq.quasinvarDef} as well as~\eqref{eq.defK}, and obtain  
 \begin{align*}
     \mathbb{E}_{\mu_t}[\|\Omega\|^2_{H^s}] - \mathbb{E}_{\mu}[\|\Omega\|^2_{H^s}] &\leqslant 2\int_0^{\infty} \lambda (\mu(A_{\lambda})^{1-\alpha t^{\eta}(1+K(t))}-\mu(A_{\lambda})) \dee \lambda \\
     & \leqslant 2\int_0^{\infty} \lambda (\mu(A_{\lambda})^{1-2\alpha t^{\eta}}-\mu(A_{\lambda})) \dee \lambda \\
     & = 2\int_0^{\infty}  \mu(A_{\lambda})\left(\exp\left(2\alpha t^{\eta}\log(\mu(A_{\lambda})^{-1}\right)-1\right) \lambda\dee \lambda. 
\end{align*}

To study the behavior of this integral as $t \to 0$ we start with the contribution $\lambda \in [0,\lambda_0]$ for which~\eqref{eq.lower2} implies: 
\[
    \int_0^{\lambda_0} \mu(A_{\lambda})\left(\exp\left(2\alpha t^{\eta}\log(\mu(A_{\lambda})^{-1})\right)-1\right) \lambda\dee \lambda \leqslant \int_0^{\lambda_0} \left(\exp(2\alpha t^{\eta}\log(c_3^{-1}\lambda^{-2}))-1\right) \lambda \dee \lambda.
\]
Writing $C=c_3^{-1}$ which we can take larger than $1$ and $\lambda_0^2$, we arrive at  
\[
    \int_0^{\lambda_0} \mu(A_{\lambda})\left(\exp\left(2\alpha t^{\eta}\log(\mu(A_{\lambda})^{-1})\right)-1\right) \lambda\dee \lambda \leqslant \int_0^{\lambda_0} \left(C^{2\alpha t^{\eta}}\lambda^{-4\alpha t^{\eta}}-1\right)\lambda\dee \lambda.
\]
Note that we only consider small values of $t$, namely $t\leqslant t_1(\alpha, \eta)$ to ensure the convergence of the integral. Let us now write 
\[
    \int_0^{\lambda_0} \left(C^{2\alpha t^{\eta}}\lambda^{-4\alpha t^{\eta}}-1\right)\lambda\dee \lambda \leqslant \lambda_0t^{\eta} \int_0^{\lambda_0} \frac{(C\lambda^{-2})^{2\alpha t^{\eta}}-1}{t^{\eta}}\dee \lambda =: \lambda_0t^{\eta} H(t^{\eta}).
\]
We immediately compute 
\[  
    H(s)=C^{2\alpha s}\frac{\lambda_0^{1-4\alpha s}}{s-4\alpha s^2} - \frac{\lambda_0}{s}
\]
which satisfies $\sup_{s\in(0,1]} |H(s)| < \infty$, and therefore 
\begin{align}
    \label{eq.ineq1}
    \int_0^{\lambda_0} \mu(A_{\lambda})\left(\exp\left(2\alpha t^{\eta}\log(\mu(A_{\lambda})^{-1})\right)-1\right) \lambda\dee \lambda \leqslant C(\eta,\alpha)t^{\eta}
\end{align}
for $0\leqslant t_1(\alpha,\eta)$. 

We turn to the estimation of the other part of the integral using~\eqref{eq.lower1} to obtain  
\begin{align*}
     \int_{\lambda_0}^{\infty} \mu(A_{\lambda})\left(\exp\left(2\alpha t^{\eta}\log(\mu(A_{\lambda})^{-1}\right)-1\right) \lambda\dee \lambda &\leqslant \int_{\lambda_0}^{\infty} e^{-c_1\lambda^2}\left(\exp\left(2c_2\alpha t^{\eta}\lambda^2\right)-1\right) \lambda\dee \lambda \\ 
     & \leqslant  \int_{0}^{\infty} e^{-c_1\lambda^2}\left(\exp\left(2c_2\alpha t^{\eta}\lambda^2\right)-1\right) \lambda\dee \lambda,
\end{align*}
integral that we split depending on whether $2c_2\alpha t^{\eta}\lambda ^2>1$ or not. Denoting $C^2=2c_2\alpha$, we have 
\begin{align*}
     \int_{0}^{\infty} e^{-c_1\lambda^2}\left(\exp\left(2c_2\alpha t^{\eta}\lambda^2\right)-1\right)& \lambda\dee \lambda \leqslant \int_{0}^{Ct^{-\eta/2}} e^{-c_1\lambda^2}\left(\exp\left(2c_2\alpha t^{\eta}\lambda^2\right)-1\right) \lambda\dee \lambda \\
     &+\int_{Ct^{-\eta /2}}^{\infty} e^{-c_1\lambda^2}\left(\exp\left(2c_2\alpha t^{\eta}\lambda^2\right)-1\right) \lambda\dee \lambda \\
     & \leqslant  2ec_2\alpha t^{\eta}\int_{0}^{Ct^{-\eta/2}}e^{-c_1\lambda_2}\lambda^3\dee \lambda + \int_{Ct^{-\eta /2}}^{\infty} e^{(2c_2\alpha t^{\eta}-c_1)\lambda^2}\lambda\dee \lambda,
\end{align*}
and we can now use 
\[
    2ec_2\alpha t^{\eta}\int_{0}^{Ct^{-\eta/2}}e^{-c_1\lambda^2}\lambda^3\dee \lambda \lesssim t^{\eta},  
\]
and that since we can restrict to small $t$ such that $2c_2\alpha t^{\eta}\leqslant c_1/2$ we have 
\begin{align*}
   \int_{Ct^{-\eta /2}}^{\infty} e^{(2c_2\alpha t^{\eta}-c_1)\lambda^2}\lambda\dee \lambda & \leqslant \int_{Ct^{-\eta /2}}^{\infty} e^{-\frac{c_1}{2}\lambda^2}\lambda \dee \lambda \\
   &\leqslant e^{-\frac{c_1C^2}{4}t^{-\eta}}\int_{Ct^{-\eta /2}}^{\infty} e^{-\frac{c_1}{4}\lambda^2}\lambda \dee \lambda \\
   &= \mathcal{O}(e^{-\frac{c_1C^2}{4}t^{-\eta}}) = \mathcal{O}(t^{\eta}). 
\end{align*}
%Combining~\eqref{eq.quasinvarDef}, the fact that $\mu(A_{\lambda})\leqslant e^{-c\lambda^2}$ and this remark, we have 
%\begin{align*}
%    \mu_t(A_{\lambda}) - \mu(A_{\lambda}) \leqslant \mu(A_{\lambda})^{(1+|t|^{\eta})^{-\alpha}}
%\end{align*}
%\[
%    \mu_t(A_{\lambda}) - \mu(A_{\lambda}) \leqslant \mu(A_{\lambda}) \left(\exp\left(\alpha |t|^{\eta}(1+K(t))\log (\mu(A_{\lambda})^{-1}\right) -1\right) \leqslant \mu(A_{\lambda}) \left(\exp\left(2\alpha |t|^{\eta}\log (\mu(A_{\lambda})^{-1}\right) -1\right).
%\]
%Using $\mu(A_{\lambda}) \leqslant e^{-c\lambda^2}$ we infer 
%\[
%    \mu_t(A_{\lambda}) - \mu(A_{\lambda}) \leqslant e^{-c\lambda ^2}\left(e^{2\alpha t^{\eta}}-  1\right)
%\]
%Up to taking smaller $t_0$ we can assume that for all $t \in [-t_0,t_0]$ there holds $(1-\alpha |t|^{\eta}(1+K(t))) \geqslant 0$, it follows from the fact that $\mu(A_{\lambda})\leqslant e^{-c\lambda^2}$ that 
%\begin{align*}
% \mu_t(A_{\lambda}) - \mu(A_{\lambda}) &\leqslant \exp\left(-c(1-\alpha |t|^{\eta}(1+K(t)))\lambda^2\right) - \mu(A_{\lambda}) \\
% & = \mu(A_{\lambda})\left(\exp\left(\alpha |t|^{\eta}(1+K(t))\log (\mu(A_{\lambda})^{-1}\right) - 1\right) \\
% & \leqslant e^{-c\lambda^2}\left(\right)
%\end{align*}
%\begin{align*}
%    \mu_t(A_{\lambda}) - \mu(A_{\lambda})&\leqslant \mu(A_{\lambda})\left(\exp\left(-((1+|t|^{\delta})^{-\alpha}-1)\log(\mu(A_{\lambda})^{-1})\right)-1 \right)\\
%    &\leqslant \mu(A_{\lambda})|(1+|t|^{\delta})^{-\alpha}-1|\log(\mu(A_{\lambda})^{-1}) \\
%    & \leqslant C \mu(A_{\lambda})\log(\mu(A_{\lambda})^{-1})|t|^{\delta}\,,
%\end{align*}
The estimate of $\EE_{\mu}[\|\Omega\|^2_{H^s}] - \EE_{\mu_t}[\|\Omega\|^2_{H^s}]$ is performed using the same techniques as above. 
\end{proof}

We postpone the proof of \cref{prop.orthogonality} to \cref{sec.ortho} and the proof of \cref{prop.estW} to \cref{sec.perturbative}, and first finish the proof of the main theorem. 

\begin{proof}[Proof of \cref{thm.main}] 
    We first use \cref{prop.orthogonality} and a Taylor expansion to write 
    \begin{align*}
        \|\Omega^{\omega}(t)-w^{\omega}(t)\|_{L^2_{\omega}H^s} &= \sqrt{\|\Omega_0^{\omega}\|_{L^2_{\omega}H^s}^2 + \gamma t^2 +\delta t^4} \\ 
        & =\|\Omega_0^{\omega}\|_{L^2_{\omega}H^s} + \frac{\gamma}{2\|\Omega_0\|_{L^2_{\omega}H^s}}t^2 + \mathcal{O}(t^4)\,,      
    \end{align*}
    so that by the triangle inequality and \cref{prop.estW} there holds: 
    \[
        \frac{\gamma}{2\|\Omega_0^{\omega}\|_{L^2_{\omega}H^s}}t^2 + \mathcal{O}(t^3)\ \leqslant \|\Omega^{\omega}(t)\|_{L^2_{\omega}H^s} - \|\Omega_0^{\omega}\|_{L^2_{\omega}H^s} \leqslant   \frac{\gamma}{2\|\Omega_0^{\omega}\|_{L^2_{\omega}H^s}}t^2 + \mathcal{O}(t^3)\,
    \]
    for $0 \leqslant t \leqslant t_1$, so that for times $0 \leqslant t \leqslant t_2=t_2(s,\sigma,(a_n)) \ll t_1$ we can write 
    \begin{equation}
        \label{eq.boundMain}
        c_1\frac{\gamma}{\|\Omega_0^{\omega}\|_{L^2_{\omega}H^s}}t^2\ \leqslant \|\Omega^{\omega}(t)\|_{L^2_{\omega}H^s} - \|\Omega^{\omega}_0\|_{L^2_{\omega}H^s} \leqslant   c_2\frac{\gamma}{\|\Omega^{\omega}_0\|_{L^2_{\omega}H^s}}t^2\,
    \end{equation}
with $(c_1,c_2)=(3/4,1/4)$ if $\gamma <0$ and $(1/4,3/4)$ if $\gamma >0$.

In the invariant case there holds 
\begin{equation}
    \label{eq.invariantEqual}
    \|\Omega^{\omega}(t)\|_{L^2_{\omega}H^s} - \|\Omega^{\omega}_0\|_{L^2_{\omega}H^s} = 0  
\end{equation}
which contradicts~\eqref{eq.boundMain}. 

In the quasi-invariant case, replace~\eqref{eq.invariantEqual} with~\eqref{eq.quasinvarB} to obtain a contradiction.
\end{proof}

\section{The orthogonality argument}\label{sec.ortho}
The goal of this section is to prove \cref{prop.orthogonality}. 

For the sake of simplicity we write $\Omega$ in place of $\Omega_0^{\omega}$. This will not lead to any confusion since the arguments in this section don't involve time evolution. We write
\[
    \Omega = \sum_{n\in\mathbb{Z}^2}a_ng_n^{\omega}e^{in\cdot x} = \frac{1}{(2\pi)^2}\sum_{n \in \mathbb{Z}^2} \hat{\Omega}_n e^{in\cdot x}.    
\] 
Similarly, let us write $U[\Omega]=(U^1,U^2)$ where $U^j=\sum_{n\in\mathbb{Z}^2}\hat{U}_n^je^{in\cdot x}$. With these notations, the relationship $\nabla \wedge U[\Omega] = \Omega$ reads $\hat{\Omega}_n= i\left(n_1\hat{U}_n^2 - n_2\hat{U}^1_n\right)$  and the relation $\nabla \cdot U[\Omega] = 0$ reads $n_1 \hat{U}^1_n + n_2 \hat{U}^2_n=0$
so that we obtain $\hat{U}^1_n = - \frac{n_2}{i|n|^2}\hat{\Omega}_n$ and $\hat{U}^2_n = \frac{n_1}{i|n|^2} \hat{\Omega}_n$.
Writing that $B_1(\Omega)=U^1\partial_1\Omega + U^2\partial_2\Omega$ we obtain 
\begin{align*}
    \widehat{B_1(\Omega)}(n) & = \sum_{k+m=n}\frac{m\cdot k^{\perp}}{|k|^2}\hat{\Omega}_k\hat{\Omega}_m  = \frac{1}{2}\sum_{k+m=n} m\cdot k^{\perp}\left(\frac{1}{|k|^2} - \frac{1}{|m|^2}\right)\hat{\Omega}_k\hat{\Omega}_m\,,
\end{align*}
and more generally
\[
    \widehat{B(\alpha, \beta)}(n)=\frac{1}{2} \sum_{k+m=n} m\cdot k^{\perp}\left(\frac{1}{|k|^2} - \frac{1}{|m|^2}\right)\hat{\alpha}_k\hat{\beta}_m\,. 
\]

The following computation is a key step in the proof of \cref{prop.orthogonality}: 

\begin{lemma}\label{lemm.computation1} Assume that $(a_n)_{n\in\mathbb{Z}^2} \in h^{s+1}$ for some $s\geqslant 0$. Let us recall that $B_2(\Omega)=B(\Omega,B_1(\Omega,\Omega))$. Then there holds 
\[
    \mathbb{E}\left[\langle \Omega, B_2(\Omega)\rangle_{H^s} \right] = \frac{1}{2(2\pi)^4}\sum_{n,q \in \mathbb{Z}^2} \langle n \rangle ^{2s}|a_n|^2|a_q|^2 (q \cdot n^{\perp})^2 \left(\frac{1}{|n|^2} - \frac{1}{|q|^2}\right)\left(\frac{1}{|q|^2}-\frac{1}{|q+n|^2}\right)\,.
\]
\end{lemma}

\begin{proof} We will decompose $B_1(\Omega)$ to obtain the claimed result. More precisely, we write: 
\begin{align*}
    \mathbb{E}\left[\langle \Omega, B_2(\Omega)\rangle_{H^s} \right] &= \frac{1}{2}\mathbb{E}\left[\langle \Omega, U[\Omega]\cdot \nabla B(\Omega,\Omega)\rangle_{H^s}\right] + \frac{1}{2}\mathbb{E}\left[\langle \Omega, U[B(\Omega,\Omega)]\cdot \nabla \Omega\rangle_{H^s}\right] \\ 
    &=: \frac{1}{2}(E_1 + E_2).
\end{align*}

Writing 
\begin{align*}
    \widehat{U[\Omega]\cdot \nabla B(\Omega,\Omega)}(n) & = \sum_{k + m = n} \frac{m\cdot k^{\perp}}{|k|^2} a_kg_k^{\omega} \widehat{B(\Omega,\Omega)}(m) \\
    & = \sum_{\substack{k + m = n \\p + q = m}} \frac{m\cdot k^{\perp}q\cdot p^{\perp}}{|k|^2|p|^2} g_{k}^{\omega}g_p^{\omega}g_q^{\omega}a_ka_pa_q,
\end{align*}
and using Plancherel's theorem, we find 
\begin{equation}
    \label{eq.E1}
    (2\pi)^4E_1 = \sum_{n \in \mathbb{Z}^2}\langle n \rangle^{2s}\sum_{\substack{k + m = n \\p + q = m}} \frac{m\cdot k^{\perp}q\cdot p^{\perp}}{|k|^2|p|^2} \bar{a}_na_ka_pa_q\mathbb{E}[\bar{g}_n^{\omega}g_{k}^{\omega}g_p^{\omega}g_q^{\omega}].
\end{equation}
Similarly 
\begin{align*}
    \mathcal{F}\left(U[B(\Omega,\Omega)]\cdot \nabla\Omega\right)(n) &= \sum_{k+m =n} \frac{m\cdot k^{\perp}}{|k|^2} \widehat{B(\Omega,\Omega)}(k)a_mg_m^{\omega} \\
    & = \sum_{\substack{k + m = n \\ p + q = k}} \frac{m\cdot k^{\perp}}{|k|^2}\frac{q \cdot p^{\perp}}{|p|^2}g^{\omega}_mg_p^{\omega}g_q^{\omega}a_ma_pa_q
\end{align*}
so that 
\begin{equation}
    \label{eq.E2}
    (2\pi)^4E_2 = \sum_{n\in \mathbb{Z}^2}\langle n \rangle^{2s} \sum_{\substack{k + m = n \\ p + q = k}} \frac{m\cdot k^{\perp}}{|k|^2}\frac{q \cdot p^{\perp}}{|p|^2} \bar{a}_na_ma_pa_q\mathbb{E}[\bar{g}_n^{\omega}g^{\omega}_mg_p^{\omega}g_q^{\omega}].
\end{equation}

Now we make the following crucial observation: because the $g_j^{\omega}$ are complex Gaussian random variables, there holds 
\[
    (q\cdot p^{\perp})\mathbb{E}[\bar{g}_n^{\omega}g_{k}^{\omega}g_p^{\omega}g_q^{\omega}] = 0
\]
unless the elements of $\{n, k, p, q\}$ are paired two-by-two (note that the cases $|n|=|k|=|p|=|q|$ are such that $q\cdot p^{\perp}=0$). There are two possible parings: the first case is $(n,k)=(p,-q)$, $n\neq \pm k$, and the second case is $(n,k)=(q,-p)$, $n\neq \pm k$. Remark that the pairing $(n,p)=(k,-q)$ is such that $q\cdot p^{\perp}=0$ therefore giving a zero contribution.

Let us write $(2\pi)^4E_1 = (2\pi)^4(\sum_{n\in\mathbb{Z}^2} \langle n \rangle^{2s}(E_{11}(n)+E_{12}(n))$ where $E_{11}(n)$ (resp.\ $E_{12}(n)$) corresponds to the first (resp.\ second) paring case. For the first contribution we find
\[
    E_{11}(n) = \sum_{q \in \mathbb{Z}^2} \frac{(q \cdot n^{\perp})^2}{|n|^2|q|^2}|a_n|^2|a_q|^2,    
\]
where we have used that $m\cdot k^{\perp}=q \cdot n^{\perp}$.

Similarly  
\[
    E_{12}(n)= -\sum_{p\in \mathbb{Z}^2} \frac{(p \cdot n^{\perp})^2}{|p|^4}|a_p|^2|a_n|^2,
\]
so that finally (re-labeling $p$ to $q$ and summing both contributions) we obtain
\[
    (2\pi)^4E_1=-\sum_{n,q \in \mathbb{Z}^2}\langle n\rangle^{2s}(q \cdot n^{\perp})^2|a_n|^2|a_q|^2 \left(\frac{1}{|q|^4}-\frac{1}{|q|^2|n|^2}\right). 
\]
Writing similarly $(2\pi)^4E_2 = (2\pi)^4(\sum_{n\in\mathbb{Z}^2} \langle n \rangle^{2s}(E_{21}(n)+E_{22}(n))$ we find 
\[
    A_{21}(n)=-\sum_{q\in \mathbb{Z}^2}\frac{(q \cdot n ^{\perp})^2}{|n+q|^2|n|^2}|a_n|^2|a_q|^2  
\]
and 
\[
    A_{22}(n)=\sum_{p\in \mathbb{Z}^2} \frac{(p\cdot n^{\perp})^2}{|p+n|^2|p|^2}|a_n|^2|a_p|^2    
\]
so that after the change of variables $(n,p) \mapsto (q,n)$ in $A_{22}$, we obtain
\[
    A_2 = -\sum_{n,q \in \mathbb{Z}^2}\langle n\rangle^{2s}(q\cdot n^{\perp})^2|a_n|^2|a_q|^2\left(\frac{1}{|q+n|^2|n|^2} - \frac{1}{|q+n|^2|q|^2}\right),    
\]
which is enough to conclude the proof.
\end{proof}

We also need the simpler computation:

\begin{lemma}\label{lemm.computation2} Assuming $(a_n)_{n\in\mathbb{Z}^2} \in h^{s+1}$ for some $s\geqslant 0$, there holds
\[
    \mathbb{E}[\|B_1(\Omega)\|_{H^s}^2] = \frac{1}{(2\pi)^4}\sum_{n,q \in \mathbb{Z}^2} \langle n + q \rangle^{2s} \frac{(q\cdot n^{\perp})^2}{2} \left(\frac{1}{|q|^2}-\frac{1}{|n|^2}\right)^2|a_q|^2|a_n|^2\,. 
\]
\end{lemma}

\begin{proof} First we write 
\[
    \|B_1(\Omega)\|_{H^s}^2 = \frac{1}{(2\pi)^4}\sum_{n\in \mathbb{Z}^2}\langle n\rangle ^{2s} \sum_{\substack{k+m=n \\ k' + m' = n}} c_{k,m}c_{k',m'}a_ka_m\bar{a}_{k'}\bar{a}_{m'}g_k^{\omega}g_m^{\omega}\bar{g}_{k'}^{\omega}\bar{g}_{m'}^{\omega} \,,
\]
where $c_{k,m}=\frac{m\cdot k^{\perp}}{2} \left(\frac{1}{|k|^2} - \frac{1}{|m|^2}\right)$.

Taking the expectation in the above and using that $c_{k,m}=0$ if $|k|=|m|$, there are two possible parings of $\{k,m,k',m'\}$ to consider: $(k,m)=(k',m')$ and $(k,m)=(m',k')$ both leading to the same contribution, therefore
\begin{align*}
    (2\pi)^4\mathbb{E}[\|B_1(\Omega)\|_{H^s}^2] &= 2\sum_{n\in \mathbb{Z}^2}\langle n \rangle ^{2s} \sum_{k+m=n}\frac{(k\cdot m^{\perp})^2}{4}\left(\frac{1}{|k|^2}-\frac{1}{|m|^2}\right)^2|a_m|^2|a_k|^2 \\ 
    & = \sum_{n,q \in \mathbb{Z}^2} \langle n + q \rangle^{2s} \frac{(q\cdot n^{\perp})^2}{2} \left(\frac{1}{|q|^2}-\frac{1}{|n|^2}\right)^2|a_q|^2|a_n|^2\,,
\end{align*}
where in the last step we just re-labeled: $(n,k,m) \mapsto (n+q,q,n)$.
\end{proof}

We go back to the notation $\Omega_0^{\omega}$ in the following proof.

\begin{proof}[Proof of \cref{prop.orthogonality}]
Let $s>0$. Expanding the squared $H^s$ norm and taking the expectation we compute 
\begin{align*}
    \|\Omega_0^{\omega} -tB_1(\Omega_0^{\omega}) + t^2B_2(\Omega_0^{\omega})\|^2_{L^2_{\omega}H^s} & = \|\Omega_0^{\omega}\|^2_{L^2_{\omega}H^s} + \gamma t^2 + \delta t^4 \\
    & -2t \mathbb{E}\left[\langle \Omega_0^{\omega},B_1(\omega_0^{\omega})\rangle_{H^s}\right] \\
    & -2t^3 \mathbb{E}\left[\langle B_1(\Omega_0^{\omega}),B_2(\omega_0^{\omega})\rangle_{H^s}\right] 
\end{align*}
where $\delta = \|B_2(\Omega_0^{\omega})\|^2_{L^2_{\omega}H^s}$ and $\gamma = \mathbb{E}\left[\|B_1(\Omega_0^{\omega})\|_{H^s}^2 + 2\langle \Omega_0^{\omega}, B_2(\Omega_0^{\omega})\rangle_{H^s}\right]$    

Using \cref{lemm.computation1,lemm.computation2} we have 
\begin{align*}
    \gamma &= \sum_{n,q \in\mathbb{Z}^2} \frac{\left(|a_n||a_q| (q \cdot n^{\perp})\right)^2}{(2\pi)^4}\left( \langle n \rangle ^{2s} \left(\frac{1}{|n|^2} - \frac{1}{|q|^2}\right)\left(\frac{1}{|q|^2}-\frac{1}{|q+n|^2}\right) + \frac{\langle n + q \rangle^{2s}}{2} \left(\frac{1}{|q|^2}-\frac{1}{|n|^2}\right)^2\right) \\ 
    & = \sum_{n,q \in\mathbb{Z}^2} b_{n,q}\left(\langle n \rangle ^{2s} \left(\frac{1}{|n|^2} - \frac{1}{|q|^2}\right)\left(\frac{1}{|q|^2}-\frac{1}{|q+n|^2}\right) + \frac{\langle n + q \rangle^{2s}}{2} \left(\frac{1}{|q|^2}-\frac{1}{|n|^2}\right)^2\right), 
\end{align*}
where $b_{n,q} = \frac{1}{(2\pi)^4}\left(|a_n||a_q| (q \cdot n^{\perp})\right)^2\left(\frac{1}{|n|^2}-\frac{1}{|q|^2}\right)$. Observe now that $b_{q,n}=-b_{n,q}$ so that using the symmetry between $n$ and $q$, we finally find  
\begin{equation*}
    \gamma = \sum_{n,q \in \mathbb{Z}^2} b_{n,q}\left( \frac{\langle n \rangle ^{2s}}{2}\left(\frac{1}{|q|^2}-\frac{1}{|q+n|^2}\right) + \frac{\langle q \rangle ^{2s}}{2} \left(\frac{1}{|q+n|^2}-\frac{1}{|n|^2}\right) + \frac{\langle n + q \rangle^{2s}}{2} \left(\frac{1}{|n|^2}-\frac{1}{|q|^2}\right)\right), 
\end{equation*}
which is~\eqref{eq.gamma}, as claimed. 

It remains to observe that $\mathbb{E}\left[\langle \Omega_0^{\omega},B_1(\omega_0^{\omega})\rangle_{H^s}\right]=0$ and $\mathbb{E}\left[\langle B_1(\Omega_0^{\omega}),B_2(\omega_0^{\omega})\rangle_{H^s}\right] =0$. In order to do so, observe that computations used in the proof of \cref{lemm.computation1} imply
\[
    \EE[\langle \Omega_0^{\omega}, B_1(\Omega_0^{\omega})\rangle_{H^s}] = \sum_{\substack{n\in\mathbb{Z}^2 \\ k+m=n}} \langle n\rangle^{2s}\frac{m\cdot k^{\perp}}{|k|^2}a_ka_m\bar{a}_n \EE[g^{\omega}_kg^{\omega}_m\bar{g}^{\omega}_n] = 0,   
\]
since $\EE[g^{\omega}_kg^{\omega}_m\bar{g}^{\omega}_n]=0$ for all values of $k, m, n \in\mathbb{Z}^2$.

it follows that $\mathbb{E}[\langle \Omega_0^{\omega}, B_1(\Omega_0^{\omega})\rangle]=0$. For the same reason, there also holds that $\EE[\langle B_1(\Omega_0^{\omega}),B_2(\Omega_0^{\omega}) \rangle_{H^s}] =0$.
\end{proof}

\section{The perturbative argument}\label{sec.perturbative}

Let us recall that $w^{\omega}(t)$, defined by 
\[
    \Omega^{\omega}(t)=\Omega^{\omega}_0 - tB_1(\Omega^{\omega}_0) + t^2B_2(\Omega^{\omega}_0) + w^{\omega}(t)\,,
\]
is the remainder in the ansatz~\eqref{eq.ansatz}. In particular $w^{\omega}(0)=0$. First, by definition of $w^{\omega}(t)$, we have
\[
    \partial_t \Omega^{\omega}(t)= - B_1(\Omega^{\omega}_0) + 2t B_2(\Omega^{\omega}_0) + \partial_tw^{\omega}(t)\,.
\] 
Expanding $B(\Omega^{\omega}(t),\Omega^{\omega}(t))$, using that $\Omega^{\omega}(t)$ solves~\eqref{eq.eulerVorticity} we find that $w^{\omega}(t)$ (which will be denoted $w$, omitting the $\omega$ superscripts) satisfies the following evolution equation:
\begin{align}
    \partial_t w(t) &= -B(w(t),w(t)) \label{eq.quadW}\\
    & - 2B(\Omega_0,w(t)) + 2t B(B_1(\Omega_0),w(t)) \label{eq.linW}  \\
    & - t^2 \underbrace{\left(B(B_1(\Omega_0),B_1(\Omega_0))  + 2B(\Omega_0,B_2(\Omega_0)) \right)}_{B_3(\Omega_0)} - 2t^2 B(B_2(\Omega_0),w(t)) \label{eq.mainRemainder} \\
    & + 2t^3 \underbrace{B(B_1(\Omega_0),B_2(\Omega_0))}_{B'_3(\Omega_0)} - t^4\underbrace{B(B_2(\Omega_0),B_2(\Omega_0))}_{\tilde B_3(\Omega_0)} \label{eq.okRemainder}.
\end{align}
Let us make a few comments:
\begin{itemize}
    \item Our proof proceeds by performing energy estimates in $H^s$, targeting an estimate of the form $\|w(t)\|_{H^s} \leqslant C t^3e^{ct}$ on short time (the exponential accounts for the use of the Grönwall inequality). One major threat to obtain such estimates will come from the random constants appearing in the estimates. For example we need to avoid factors like $e^{\norm{\Omega_0}_{H^\sigma}^3}$ appearing in our final estimate (which would not be integrable with respect to $\mu$). In particular, we have to pay attention when estimating~\eqref{eq.linW}. We want an estimate of the form $\|w(t)\|_{H^s} \leqslant Ct^3$, we expect that~\eqref{eq.mainRemainder} will be part of the main term in the estimate, but that~\eqref{eq.okRemainder} will contribute negligibly.  
    \item We point out that the usual energy estimate on the term $B(w(t),w(t))$ would bring a potential double-exponential bound, which on short-time is bounded. In our case we must take advantage of $w(0)=0$, so that the nonlinear term $B(w(t),w(t))$ should be viewed as a small contribution. 
    \item Because the method we employ is to use the ansatz~\eqref{eq.ansatz} for $\Omega^{\omega}(t)$ we will need to estimate $B_3(\Omega_0), \tilde{B}_3(\Omega_0), B'_3(\Omega_0)$, and it is the reason why we require $\sigma >s+3$.
\end{itemize}

\begin{proof}[Proof of \cref{prop.estW}] In this proof we only work with $t \leqslant 1$, and we will later restrict to shorter intervals of time. 

We start by deriving a first \textit{a priori} bound for $w^{\omega}(t)$ by using the triangle inequality, \cref{lemm.bilinearForm} and \cref{prop.globalBound}, 
\begin{align}\label{eq.aprioriW}
    \|w^{\omega}(t)\|_{H^{s'}} &\leqslant \|\Omega^{\omega}(t)\|_{H^{s'}} + \|\Omega^{\omega}_0\|_{H^{s'}} + t\|B_1(\Omega^{\omega}_0)\|_{H^{s'}} + t^2\|B_2(\Omega^{\omega}_0)\|_{H^{s'}} \\
    & \leqslant C_{\omega}\sqrt{\log(2+t)} + C(1+t+t^2) \left(\|\Omega^{\omega}_0\|_{H^{s'}} + \|\Omega^{\omega}_0\|^2_{H^{s'+1}} + \|\Omega^{\omega}_0\|_{H^{s'+2}}^3\right) \notag  \\
    & \leqslant C(C_{\omega} + (1+\|\Omega_0^\omega\|_{H^{s'+2}}^3))=:D_{\omega} \notag 
\end{align}
for all $0 \leqslant t \leqslant 1$ and all $s'\leqslant \sigma - 2$. 

We will later use that the random constant $D_{\omega}$ lies in $L^p_{\omega}$ for all $1\leqslant p< \infty$. However, note that it lacks exponential moments, therefore directly plugging this estimate in~\eqref{eq.quadW} and performing usual energy estimates would yield a bound on $\|w(t)\|_{H^s}$ which would not lie in $L^2_{\omega}$.

Let us now move to the energy estimate. We compute $\frac{1}{2}\frac{\mathrm{d}}{\mathrm{d}t}\|w(t)\|^2_{H^s} = (\partial_t w(t),w(t))_{H^s}$ using~\eqref{eq.quadW} to decompose $\partial_t w(t)$. 

First, we use \cref{lemm.bilinearForm} to estimate~\eqref{eq.quadW} by
\begin{align*}
    \label{eq.quadraticBound}
    \left\vert(B(w(t),w(t)),w(t))_{H^s}\right\vert &\leqslant C(s,\varepsilon)\|w(t)\|_{H^s}^2\|w(t)\|_{H^{1+\varepsilon}},
\end{align*}
where $C(s,\varepsilon)$ is a global deterministic constant.  

For the linear terms in~\eqref{eq.linW} we rely on \cref{lemm.bilinearForm} again to bound
\[
    \left(B(\Omega_0,w(t)),w(t)\right)_{H^s} \leqslant C(s,\varepsilon)\|\Omega_0\|_{H^{s+1+\varepsilon}}\|w(t)\|_{H^{s}}^2\,,
\]
and
\[
    \left(B(B_1(\Omega_0),w(t)),w(t)\right)_{H^s} \leqslant C(s,\varepsilon)\|B_1(\Omega_0)\|_{H^{s+1+\varepsilon}}\|w(t)\|_{H^s}^2\,, 
\]
which together imply 
\begin{equation}
    \label{eq.linearBound}
    \left(-2B(\Omega_0,w(t)) +2t B(B_1(\Omega_0),w(t)), w(t)\right)_{H^s} \leqslant C(s,\varepsilon)(1+\|\Omega_0\|_{H^{s+2+\varepsilon}}^2) \|w(t)\|^2_{H^s}.
\end{equation}

To estimate~\eqref{eq.mainRemainder} we use the Cauchy-Schwarz inequality and \cref{lemm.bilinearForm} to obtain 
\begin{align*}
    \left(B_3(\Omega_0),w(t)\right)_{H^s} & \leqslant \|B_3(\Omega_0)\|_{H^{s}}\|w(t)\|_{H^s} \leqslant C(s,\varepsilon) \|\Omega_0\|_{H^{s+3}}^4 \|w(t)\|_{H^s}.
\end{align*} 
Another application of \cref{lemm.bilinearForm} yields 
\begin{align*}
       \left(B(B_2(\Omega_0),w(t)),w(t)\right)_{H^s} & \leqslant \|B_2(\Omega_0)\|_{H^{s+1+\varepsilon}}\|w(t)\|_{H^s} \\
       &\leqslant C(s,\varepsilon)\|\Omega_0\|^3_{H^{s+3+\varepsilon}}\|w(t)\|_{H^s}^2 \\
       &\leqslant D_{\omega}\|\Omega_0\|^3_{H^{s+3+\varepsilon}}\|w(t)\|_{H^s}\,
\end{align*}
where in the last inequality we have used~\eqref{eq.aprioriW}. 
Therefore  
\begin{equation}
    \label{eq.mainBound}
    \left(-t^2B_3(\Omega_0)-2t^2B(B_2(\Omega_0),w(t)),w(t)\right)_{H^s} \leqslant D_{\omega}t^2(1+\|\Omega_0\|^4_{H^{s+3+\varepsilon}})\|w(t)\|_{H^s}\,.
\end{equation}

Finally, also using \cref{lemm.bilinearForm} we can also bound the remainder terms~\eqref{eq.okRemainder} and obtain 
\begin{equation}
    \label{eq.otherBound}
    \left(2t^3B'_3(\Omega_0)-t^4\tilde{B}_3(\Omega_0),w(t)\right)_{H^s} \leqslant C(s,\varepsilon)t^3(1+\|\Omega_0\|^6_{H^{s+3}})\|w(t)\|_{H^s}.
\end{equation} 

Combining~\eqref{eq.linearBound}, \eqref{eq.mainBound} and \eqref{eq.otherBound} we finally obtain 
\begin{align}
    \frac{\mathrm{d}}{\mathrm{d}t} \|w(t)\|^2_{H^s} &\leqslant \|w(t)\|^2_{H^s}\|w(t)\|_{H^{1+\varepsilon}} + C(s,\varepsilon)(1+\|\Omega_0\|_{H^{s+2+\varepsilon}}^2) \|w(t)\|^2_{H^s} \\
    & + D_{\omega}t^2(1+\|\Omega_0\|^4_{H^{s+3 + \varepsilon}})\|w(t)\|_{H^s} + C(s,\varepsilon)t^3(1+\|\Omega_0\|^6_{H^{s+3}})\|w(t)\|_{H^s}.  
\end{align}
Using $t\leqslant 1$, dividing by $\|w(t)\|_{H^s}$ and collecting the constants we find 
\begin{equation}
    \label{eq.diffInequal}
    \frac{\mathrm{d}}{\mathrm{d}t} \|w(t)\|_{H^s} \leqslant \|w(t)\|_{H^s}\|w(t)\|_{H^{1+\varepsilon}} +  k_{\omega} \|w(t)\|_{H^s} + K_{\omega}t^2
\end{equation}
where $\mathbb{E}[e^{ck_{\omega}}]<\infty$ for $c$ small enough and $K_{\omega} \in L^p_{\omega}$ for all finite $p$. 

Let us now use~\eqref{eq.aprioriW} for $0 \leqslant t \leqslant 1$ in ~\eqref{eq.diffInequal} to bound 
\[
    \frac{\mathrm{d}}{\mathrm{d}t} \|w(t)\|_{H^s} \leqslant D'_{\omega}\,,
\]
where $D'_{\omega}\in L^p_{\omega}$ for all finite $p$. Note that we have used that $1+\varepsilon \leqslant \sigma -2$ in order to use~\eqref{eq.aprioriW}. Integrating this differential bound yields $\|w(t)\|_{H^s}  \leqslant D'_{\omega}t$ for $0\leqslant t \leqslant 1$. Using the latter inequality (also used with $s=1+\varepsilon$)  in~\eqref{eq.diffInequal}, we obtain the differential inequality
\[
    \frac{\mathrm{d}}{\mathrm{d}t} \|w(t)\|_{H^s} \leqslant k_{\omega} \|w(t)\|_{H^s} + K'_{\omega}t^2,
\]
where $K'_{\omega} \coloneqq (D'_{\omega})^2 + K_{\omega}$. Therefore the Grönwall inequality implies 
\[
    \|w(t)\|_{H^s} \leqslant \frac{1}{3}K'_{\omega}t^3 \exp(k_{\omega}t) \leqslant \frac{1}{3}K'_{\omega}t^3 \exp(\frac{c}{2}k_{\omega})
\]
if $0 \leqslant t \leqslant \frac{c}{2}$. Therefore $\|w(t)\|_{L^2_{\omega}H^s}\leqslant \frac{1}{3}\mathbb{E}[K_{\omega}^2\exp(ck_{\omega})]^{\frac{1}{2}}t^3$, and the claim is proven with $C=\frac{1}{3}\mathbb{E}[(K'_{\omega})^2\exp(ck_{\omega})]^{\frac{1}{2}}<\infty$.
\end{proof}

\section{Genericity of the condition \texorpdfstring{$\gamma >0$}{gamma}}\label{sec.genericity}

\subsection{First examples of non-invariant Gaussian measures}

We start with providing a very simple example of $(a_n)_{n\in\mathbb{Z}^2}$ for which $\gamma_{s,(a_n)}\neq 0$. 

\begin{lemma}\label{lem.example1} Let $a_n=1$ if $n\in\{(1,0),(-1,0),(0,2),(0,-2)\}$, and $a_n=0$ otherwise. If $s >1$ then $\gamma_{s,(a_n)} \neq 0$. Therefore the associated measure is not-invariant under the flow of~\eqref{eq.eulerVorticity} (nor quasi-invariant under \cref{assump.quasinvar}.)
\end{lemma}

\begin{proof} We can compute explicitly $\gamma_{s,(a_n)}$. Using the notation of~\eqref{eq.gamma} we find: 
\[
    \gamma = 3\left(\frac{6^{s}}{2} + \frac{3 \cdot 2^{s}}{10} - \frac{4 \cdot 3^{s}}{5}\right)=\frac{3}{20}\left(15\cdot 6^{s} - 16\cdot 5^{s} + 2^{s}\right)\,,
\]
and we need to check that this does not vanish. Note that $s > 1$ we observe that $\left(\frac{6}{5}\right)^{s}>\frac{6}{5} > \frac{16}{15}$ which implies $15\cdot 6^{s} - 16\cdot 5^{s}>0$ so that $\gamma >0$. 
\end{proof}

As a corollary for all $\sigma >3$, one can construct a sequence $(a_n)_{n\in\mathbb{Z}^2}$ such that the associated measure $\mu_{(a_n)}$ is `exactly' supported on $H^{\sigma}$ and $\gamma_{s,(a_n)} \neq 0$ for all $s\in (0,\sigma]$.  

\begin{corollary} Let $\sigma > 3$. There exists $\mu \in \mathcal{M}^{\sigma}$ which is not invariant under the flow of~\eqref{eq.eulerVorticity} and which satisfies $\mu(H^{\sigma}) =1$, and $\mu(H^{\sigma +\varepsilon})=0$ for all $\varepsilon >0$.    
\end{corollary}

\begin{proof} We first remark that if $(a_n)_{n\in\mathbb{Z}^2} \in h_{\operatorname{rad}}^{\sigma} \setminus \bigcup_{\varepsilon >0}h_{\operatorname{rad}}^{\sigma +\varepsilon}$ with non vanishing $a_n$, then the associate measure $\mu_{(a_n)}$ satisfies $\mu(H^{\sigma}) =1$, and $\mu(H^{\sigma +\varepsilon})=0$. That $\mu(H^{\sigma}) =1$ holds follows from~\eqref{eq.HsigSupport}. The fact that for any $\varepsilon >0$ there holds $\mu(H^{\sigma +\varepsilon})=0$ is a more difficult but a standard fact, we refer to the Appendix of~\cite{burqTzvetkov} for a detailed proof. 

Let $b_n=\frac{1}{\langle n \rangle ^{\sigma +1}\log(\langle n\rangle)}$, which belongs to $h_{\operatorname{rad}}^{\sigma} \setminus \bigcup_{\varepsilon >0}h_{\operatorname{rad}}^{\sigma +\varepsilon}$. Let also $(a_n)_{n\in\mathbb{Z}^2}$ be the sequence from \cref{lem.example1} and let us consider $c_n=a_n+\varepsilon b_n$. Then by expanding $\gamma_{s,(c_n)}$ we can check that $\gamma_{s,(c_n)} = \gamma_{s,(a_n)} + \mathcal{O}(\varepsilon)$, and since $\gamma{s,(a_n)} >0$, we can take $\varepsilon$ small enough so that $\gamma_{s,(c_n)}>0$.
\end{proof}

\subsection{Proof of \cref{thm.genericity}}

The proof essentially follows from the following continuity lemma. 

\begin{lemma}\label{lemm.continuity} Let $s \geqslant 0$ and $\sigma \geqslant s+1$. The map $\gamma : (a_n)_{n\in\mathbb{Z}^2} \longmapsto \gamma_{s,(a_n)}$, where $\gamma_{s,(a_n)}$ is defined by~\eqref{eq.gamma} is continuous as an operator $h^{\sigma}(\mathbb{Z}^2) \to \mathbb{R}$. 
\end{lemma}

\begin{proof} Let us prove the continuity using the sequential characterization: we take $a^k=(a^k_n)_{\in\mathbb{Z}^2}$ converging to $a=(a_n)_{\in\mathbb{Z}^2}$ in $h^{\sigma}(\mathbb{Z}^2)$. Then let us write 
\[
    \gamma_{a^k} = \sum_{n,q \in \mathbb{Z}^2} |a^k_n|^2|a^k_q|^2 \alpha_{n,q}  
\] 
where $\alpha_{n,q}=(q\cdot n^{\perp})^2\left(\frac{1}{|n|^2} - \frac{1}{|q|^2}\right)\beta_{n,q} = \mathcal{O}(\langle n\rangle^{2(s+1)} \langle q\rangle^{2(s+1)})$, where $\beta_{n,q}$ is defined in~\eqref{eq.beta}. This crude bound will be enough for our purpose. Using the identity
\[
    |a_n^k|^2|a_q^k|^2-|a_n|^2|a_q|^2 = (|a_n^k||a_q^k| - |a_n||a_q|)(|a_n^k||a_q^k| + |a_n||a_q|)\,,    
\]
and using the Cauchy-Schwarz inequality we infer 
\begin{align*}
    |\gamma_{a^k}-\gamma_a|^2 &\leqslant \sum_{n,q \in \mathbb{Z}^2} (|a_n^k||a_q^k| - |a_n||a_q|)^2\langle n\rangle^{2(s+1)} \langle q\rangle^{2(s+1)} \\ 
    & \times \sum_{n,q \in \mathbb{Z}^2} (|a_n^k||a_q^k| + |a_n||a_q|)^2 \langle n\rangle^{2(s+1)} \langle q\rangle^{2(s+1)}\,.
\end{align*}
First, we bound 
\[
    \sum_{n,q \in \mathbb{Z}^2} (|a_n^k||a_q^k| + |a_n||a_q|)^2 \langle n\rangle^{2(s+1)}\langle q\rangle^{2(s+1)} \leqslant \|a^k\|_{h^{s+1}}^4 + \|a\|_{h^{s+1}}^4 
\]
which is bounded (in $k$) because $s+1 \leqslant \sigma$. For the other term we write
\[
    (|a_n^k||a_q^k| - |a_n||a_q|)^2 \lesssim |a_q|^2|a_n^k-a_n|^2 + |a_n^k|^2|a_q^k-a_q|^2\,,
\]
so that we obtain 
\begin{align*}
    \sum_{n,q \in \mathbb{Z}^2} (|a_n^k||a_q^k| - |a_n||a_q|)^2\langle n\rangle^{2(s+1)}\langle q\rangle^{2(s+1)} &\lesssim \|a\|_{h^{s+1}}^2\|a^k-a\|_{h^{s+1}}^2 + \|a_k\|_{h^{s+1}}^2\|a^k-a\|_{h^{s+1}}^2 \\
    & =\mathcal{O}(\|a^k-a\|_{h^{\sigma}}^2)  
\end{align*}
which goes to zero because $s+1\leqslant \sigma$, therefore proving the continuity.
\end{proof}

The continuity in the first part of \cref{thm.genericity} now follows from \cref{lemm.continuity} and the fact that $A^{\sigma}=\gamma^{-1}(\mathbb{R}\setminus\{0\})$, therefore open as the reciprocal image of an open set by a continuous function. Also, \cref{lemm.continuity} holds with $h^{\sigma}_{\operatorname{rad}}$ in place of $h^{\sigma}$. 

Finally let us explain why $A^{\sigma}$ has empty interior in $h^{\sigma}$ (from which the corresponding result on $\mathcal{A}^{\sigma}$ will follow). Let $(a_n)_{n\in\mathbb{Z}^2}$ such that $\gamma_{(a_n)}=\gamma_{s,(a_n)}= 0$, and let $(b_n)_{n\in\mathbb{Z}^2}$ radially symmetric such that $\gamma_{(b_n)} = \gamma_{s,(b_n)} \neq 0$, such as constructed in \cref{lem.example1}. One can therefore consider $c_n=a_n+\varepsilon b_n$ for $\varepsilon \ll 1$. There holds $\|c-a\|_{h^{\sigma}} = \mathcal{O}(\varepsilon)$, and expanding $\gamma_{(c_n)}$ in terms of $\varepsilon_n$ we find 
\[
    \gamma_{(c_n)} = \gamma_{(a_n)} +\kappa_1 \varepsilon + \kappa_2 \varepsilon^2 + \kappa_3 \varepsilon^3 + \varepsilon^4 \gamma_{(b_n)}
\]
for some real constants $\kappa_1, \kappa_2$ and $\kappa_3$. By assumption, there holds $\gamma_{(a_n)}=0$. If $\kappa_1 \neq 0$ then one just needs to take $\varepsilon$ small enough so that $|\kappa_2 \varepsilon^2 + \kappa_3 \varepsilon^3 + \varepsilon^4 \gamma_{(b_n)}| \leqslant \frac{\kappa_1\varepsilon}{2}$ and therefore $\gamma_{(c_n)} \neq 0$. If $\kappa_1=0$ we can do the same analysis on $\kappa_2$ and so on, so that the only remaining case is $\kappa_1=\kappa_2=\kappa_3=0$ in which case we use that $\gamma_{(b_n)} \neq 0$.

Continuity in \cref{thm.genericity} Part~(ii) directly follows from~\eqref{eq.KantorovichRubinstein}:  
let $a^k=(a^k_n)_{n\in\mathbb{Z}^2}$ which converges to $a=(a_n)_{n\in\mathbb{Z}^2}$ in $h^{\sigma}$ (resp.\ $h^{\sigma}_{\operatorname{rad}}$) and let $\mu_k\coloneqq\mu_{a^k}$, $\mu\coloneqq\mu_a$. We can compute for any $1$-Lipschitz function $\varphi$, we compute 
\begin{align*}
    \left\vert\int_{H^{\sigma}} \varphi(u)\mathrm{d}(\mu-\mu_k)(u) \right\vert & = \left\vert\EE\left[\varphi\left(\sum_{n\in\mathbb{Z}^2} a_n^kg_n^{\omega}e^{in\cdot x}\right) -\varphi\left(\sum_{n\in\mathbb{Z}^2} a_ng_n^{\omega}e^{in\cdot x}\right)\right]\right\vert\\
    & \leqslant \EE\left[\left\| \sum_{n\in\mathbb{Z}^2} a_n^kg_n^{\omega}e^{in\cdot x} - \sum_{n\in\mathbb{Z}^2} a_ng_n^{\omega}e^{in\cdot x}\right\|_{H^{\sigma}(\mathbb{T}^2)}\right] \\
    & \leqslant  \left\| \sum_{n\in\mathbb{Z}^2} a_n^kg_n^{\omega}e^{in\cdot x} - \sum_{n\in\mathbb{Z}^2} a_ng_n^{\omega}e^{in\cdot x}\right\|_{L^{2}_{\omega}H^{\sigma}}\,, 
\end{align*}
where in the first line of the above computation we have used the push-forward definition of the measures $\mu_k$ and $\mu$. Taking the supremum over such functions $\varphi$ yields 
\[
    W_1(\mu,\mu_k) \leqslant \|a^k-a\|_{H^{\sigma}}
\] 
which goes to zero as $k\to \infty$.

\subsection{Proof of \cref{thm.probaGenericity}}

Our main tool is a generalization of the Kakutani criterion. Kakutani~\cite{kakutani} characterized the equivalence between tensor products of equivalent measures. For Gaussian measures in Hilbert space, we use the following generalization due to Feldman and H\'ajek: 

\begin{theorem}[Feldman-H\'ajek \cite{bogachev}*{Theorem 2.7.2}] Let $\mu$ and $\nu$ be Gaussian random variables on a Hilbert space $X$ with means $m_{\mu}, m_{\nu}$ and co-variance operators $C_{\mu}, C_{\nu}$. Then $\mu$ and $\nu$ are either equivalent or mutually singular, and equivalence holds if and only if the three following conditions are met: 
\begin{enumerate}
    \item $\mu$ and $\nu$ have the same Cameron-Martin space $H=C_{\mu}^{\frac{1}{2}}(X)=C_{\nu}^{\frac{1}{2}}(X)$.
    \item $m_{\mu}-m_{\nu} \in H$.
    \item $(C_{\mu}^{-\frac{1}{2}}C_{\nu}^{\frac{1}{2}})(C_{\mu}^{-\frac{1}{2}}C_{\nu}^{\frac{1}{2}})^*- I$ is a Hilbert-Schmidt operator. 
\end{enumerate}
\end{theorem}

We will only use criterion (3) to prove that two given Gaussian measures are mutually singular. In particular as we consider very simple Gaussian measures on the Hilbert space $H^{\sigma}(\mathbb{T}^2)$ with diagonal co-variance operators, we obtain the following useful corollary.  

\begin{corollary}\label{thm.feldmanHajek} Write $e_n(x)=e^{in\cdot x}$, and let us consider the two Gaussian measures 
\[
    \mu = \sum_{n \in \mathbb{Z}^2} h_ne_n \text{ and } \nu = \sum_{n \in \mathbb{Z}^2} \tilde{h}_ne_n 
\] 
on $H^{\sigma}(\mathbb{T}^2)$, where $h_k, \tilde{h}_k$ are centered Gaussian variables of variances $\sigma_k, \tilde{\sigma}_k$. If 
\[
    \sum_{n\in\mathbb{Z}^2} \left(\frac{\tilde{\sigma}_n^2}{\sigma_n^2}-1\right)^2 = \infty\,,
\]
then $\mu$ and $\nu$ are mutually singular. 
\end{corollary}

In order to prove \cref{thm.probaGenericity}, let us fix a sequence $(a_n)_{n\in\mathbb{Z}^2}$ such that its associated measure belongs to $\mathcal{M}^{\sigma}_{\operatorname{rad}}$ and such that $\gamma_{s,(a_n)} \neq 0$. Let us consider a sequence $(\varepsilon^{\xi}_n)_{n\in \mathbb{Z}^2}$ of independent and identically distributed Bernoulli random variables satisfying $\mathbb{P}(\xi : \varepsilon^{\xi}_n = 1)=\mathbb{P}(\xi : \varepsilon^{\xi}_n = -1) = \frac{1}{2}$. We introduce
\[
    b^{\xi}_n=a_n\left(1+\frac{\varepsilon^{\xi}_n}{2\langle n\rangle}\right)^{\frac{1}{2}},
\]
and consider the associated measures $\mu_{\xi}\coloneqq\mu_{(b^{\xi}_n)}$. We check that there is a set $\Xi_1$ such that $\mathbb{P}(\xi \in \Xi_1)>0$ and for all $\xi \in \Xi_1$, there holds $\gamma_{(b^{\xi}_n)}\neq 0$. Indeed, there holds 
\[
    \mathbb{E}[\gamma_{s,(b^{\xi}_n)}] = \sum_{n, q \in \mathbb{Z}^2}\mathbb{E}[|b_n^{\xi}|^2|b_q^{\xi}|^2]  \beta_{n,q}.    
\]
Since $\beta_{n,q}=0$ when $n=q$, it follows that  
\[
    \mathbb{E}[|b_n^{\xi}|^2|b_q^{\xi}|^2] = |a_n|^2|a_q|^2\mathbb{E}\left[1+\frac{\varepsilon^{\xi}_n}{\langle n\rangle}\right]\mathbb{E}\left[1+\frac{\varepsilon^{\xi}_q}{\langle q\rangle}\right] = |a_n|^2|a_q|^2
\]
so that $\mathbb{E}[\gamma_{s,(b^{\xi}_n)}] = \gamma_{s,(a_n)} \neq 0$, by assumption, therefore $\gamma_{s,(b^{\xi}_n)} \neq 0$ for $\xi \in \Xi_1$ of positive measure. 

Let us fix $\xi_1 \in \Xi_1$, and apply \cref{thm.feldmanHajek} to $\mu_{\xi_1}$ and $\mu_{\xi}$ for some $\xi \in \Xi$ (the probability space). We write that: 
\begin{align}
    \label{eq.termSeries}
    \left(\frac{1+\frac{\varepsilon_n^{\xi_1}}{2\langle n\rangle}}{1+\frac{\varepsilon_n^{\xi}}{2\langle n\rangle}}-1\right)^2 & = \frac{1-\varepsilon_n^{\xi_1}\varepsilon_n}{2\langle n \rangle ^2} + \mathcal{O}(|n|^{-3})\\
    & = \frac{1}{2\langle n\rangle^2} - \frac{\varepsilon_n^{\xi_1}\varepsilon_n^{\xi}}{2\langle n \rangle^2} + \mathcal{O}(|n|^{-3}). \notag
\end{align}
The series $\sum_{n\in\mathbb{Z}^2} \frac{1}{\langle n \rangle ^2}$ diverges, and the series of terms $\mathcal{O}(|n|^{-3})$ converge. Finally, the random (in $\xi$) series $\sum_{n\in\mathbb{Z}^2} \frac{\varepsilon_n^{\xi_1}\varepsilon_n^{\xi}}{\langle n\rangle^{2}}$
is convergent for almost every $\xi \in \Xi$ (let $\Xi_*$ be this set of $\xi$) by application of Kolmogorov's three series Theorem (for instance we refer to \cite{durett}*{Section 1.8}). 
Therefore, we obtain that the series of general term~\eqref{eq.termSeries} is divergent for all $\xi \in \Xi_*$, which is of probability one. This yields uncountable many measures $\mu_{\xi}$ which are mutually absolutely singular with respect to $\mu_{\xi_1}$. Note that this does not yield a set of measures which are pairwise mutually singular. The end of the proof proceeds by contradiction: assume that there are only a countable number of measures $\{\mu_k\}_{k\geqslant 1}$ which are pairwise mutually singulars and such that for each $\mu_k$, $\gamma \neq 0$. Let $\Xi_{(k)}$ be the set of $\xi$ such that $\mu_{\xi} \perp \mu_{k}$. We have shown that $\mathbb{P}(\xi \in \Xi_1 \cap \Xi_{(k)}) = \mathbb{P}(\xi \in \Xi_1)>0$. 
Therefore, $\mathbb{P}(\xi \in \Xi_1 \cap \bigcap_{k\geqslant 1} \Xi_{(k)}) = \mathbb{P}(\xi \in \Xi_1) >0$. Observe now that for all $\xi \in \Xi_1 \cap \bigcap_{k\geqslant 1} \Xi_{(k)}$ there holds $\mu_{\xi} \perp \mu_k$ for all $k\geqslant 1$ and all $\gamma_{(b_n^{\xi})} \neq 0$, therefore a contradiction. 

\subsection{Proof of \cref{thm.examples} (\textit{i})}

Let us assume that $\mu$ has a Fourier support 
\[
    S=\{n\in\mathbb{Z}^2, \quad a_n \neq 0\}
\] 
which is finite.

Let us follow the terminology in~\cite{elgindiHuSverak} and say that a pair $(n,q)$ is degenerate if $n$ and $q$ are co-linear (that is the line passing through $n$ and $q$ is passing through the origin) or belong to the same circle centered at $0$. In other words, $(n,q)$ is degenerate if and only if $(n\cdot q^{\perp})\left(\frac{1}{|n|^2}-\frac{1}{|q|^2}\right)=0$.

If there exists $s\in (0,\sigma]$ such that $\gamma_s\neq 0$, there is nothing to prove. Otherwise, let us assume that $\gamma_s = 0$ for all $s \in (s_0-\varepsilon, s_0+\varepsilon)$ for some $s_0\in (0,\sigma)$. Differentiating the equality $\gamma_s = 0$ with respect to $s$ and summing linear combinations of iterated derivatives of this equality, we find that for any polynomial $P \in \mathbb{R}[X]$ there holds
\begin{align}
    0=\sum_{|n|,|q|} & |a_n|^2|a_q|^2(n\cdot q^{\perp})^2\left(\frac{1}{|n|^2}-\frac{1}{|q|^2}\right) \Biggl(P\left(2 \log (\langle n+q \rangle)\right)\langle n+q \rangle^{2s}\left(\frac{1}{|n|^2}-\frac{1}{|q|^2}\right)\label{eq.equalityPol} \\
    & + P\left(2 \log (\langle n\rangle)\right)\langle n\rangle^{2s}\left(\frac{1}{|q|^2}-\frac{1}{|n+q|^2}\right) + P\left(2 \log (\langle q\rangle)\right)\langle q \rangle^{2s}\left(\frac{1}{|n+q|^2}-\frac{1}{|n|^2}\right)\Biggr)\notag 
\end{align}
Let $M=\max\{|n|, a_n\neq 0\}$ and let $C_M\coloneqq\{n \in \mathbb{Z}^2, |n|=M \text{ and } a_n\neq 0\}$. Let $P$ be a polynomial such that $P(2\log(\langle k\rangle))=0$ for any $k\leqslant M$. and $P(2\log(\langle k\rangle))=1$ if $M<k\leqslant 2M$. Using this polynomial in~\eqref{eq.equalityPol} yields
\begin{equation}
    \label{eq.degenerate}
    \sum_{n,q: |n+q|>M}\langle n+q\rangle^{2s}|a_n|^2|a_q|^2(n\cdot q^{\perp})^2\left(\frac{1}{|n|^2}-\frac{1}{|q|^2}\right)^2 = 0,  
\end{equation}
from which we infer that if $|n+q|>M$ and $a_n, a_q \neq 0$, then $(n,q)$ is degenerate. 

We distinguish between the case $|C_M|>2$ and the case $|C_M| \leqslant 2$.

Assume that $|C_M|>2$. We claim that all the modes $n$ such that $a_n\neq 0$ are included in the circle $\{|n|=M\}$. In order to obtain this result, let us observe that since $|C_M|>2$ and that $C_M$ is symmetric (by assumption on the $a_n$) it means that there exists at least $p_1,p_2\in C_M$ such that $p_2$ does not belong to the line passing through $0$ and $p_1$. Let $q \neq p_1,p_2$ such that $a_q \neq 0$. Since $a_{-q}=a_q \neq 0$, and that either $|p_i+q|>M$ or $|p_i-q|>M$, there is no loss in generality in assuming $|p_i+q|>M$ so that for $i=1,2$, $(p_i,q)$ is degenerate in view of~\eqref{eq.degenerate}. Since this holds for $p_1$ and $p_2$ not belonging to the same line passing through $0$ this implies that $q$ belongs to the circle of radius $M$ centered at $0$. 

Let us assume $|C_M|\leqslant 2$, which actually means $|C_M|=2$ and $C_M=\{p,-p\}$. Let us prove that any $q$ such that $a_q \neq 0$ belongs to the line passing through $0$ and $p$. Let $q$ such that $|q|<M$ and $a_q\neq 0$. Without loss of generality we can assume $|p+q|>M$ (otherwise there holds $|-p+q|>M$ and replace $p$ by $-p$). Using~\eqref{eq.degenerate} it follows that $(p,q)$ is degenerate, and since $|q|<M=|p|$ it follows that $q$ belongs to the line passing through $0$ and $p$ as claimed.

\begin{remark} Combining this result with \cref{thm.genericity}, it holds that for any sequence $(a_n)_{n\in\mathbb{Z}^2}$ with compact support which is not degenerate, there exists an open neighborhood of $\mu_{(a_n)}$ in $\mathcal{M}^{\sigma}$ such that any measure in this neighborhood is non-invariant under the flow of~\eqref{eq.eulerVorticity} (nor quasi-invariant under \cref{assump.quasinvar}).
\end{remark}

\subsection{Proof of \cref{thm.examples} \textit{(ii)}} 

In this section, we prove that the Gaussian measure on $H^4(\mathbb{T}^2)$ induced by  
\[
    \sum_{n\in\mathbb{Z}^2 \setminus\{0\}} \frac{g_n^{\omega}}{\langle n\rangle^{5} \log(3+\langle n \rangle^2)}e^{in\cdot x}\,,
\]
satisfies $\gamma_{\frac{1}{2}} = \gamma > 0$ (in this section we are choosing $s=\frac{1}{2}$ and henceforth drop the subscript from $\gamma$).
Such an example is the typical example one might consider to produce a measure whose support is in $H^{4} $ but not in $\bigcup_{\varepsilon >0}H^{4 +\varepsilon}$.

Our proof that $\gamma > 0$ is computer assisted using interval arithmetic (see \cite{Tucker2011} for an introduction to mathematically rigorous computation using interval arithmetic, which fully accounts for CPU rounding errors).

We first compute for some $N>0$ the partial sum: 
\[
    \gamma_N \coloneqq \sum_{|n|,|q|>N}\alpha_{n,q} \coloneqq \sum_{|n|,|q|< N} |a_n|^2|a_q|^2(q\cdot n^{\perp})^2\left(\frac{1}{|n|^2}-\frac{1}{|q|^2}\right) \beta_{q,n},
\]
where $\beta_{q,n}=\beta_{q,n,s}$ is given by~\eqref{eq.beta}. Then we separately bound the error $|\gamma - \gamma_N|$ analytically. 
More precisely, the numerical approximation of $\gamma_N$ for $N = 30$ places it in an interval sufficiently far from zero: 
\[
    \left\{
        \begin{array}{rcl}
            \frac{1}{2}\gamma _N &\in& [0.00011184535610465990373, 0.00011184535613147070557] \\
            \frac{1}{2}|\gamma - \gamma _N| &\in & [0,0.00010534979423897216787].
        \end{array}
    \right.    
\]

The interval for $\gamma_N$ is obtained using interval arithmetic implemented in Python with the \texttt{mpmath} package\footnote{\url{https://mpmath.org}}; See \cref{app.code} for the code used to carry out the computation. Computations were carried out with \texttt{python 3.10.6} and the CPU is \texttt{11th Gen Intel(R) Core(TM) i5-1135G7 \@ 2.40GHz.}

Next, let us consider the analytic estimate on $\abs{\gamma - \gamma_N}$. 
First, note that as the summation in $\gamma$ is symmetric in $n$ and $q$, there holds
\[
    \gamma_N=\sum_{|n|,|q|< N} \alpha_{n,q}=2\sum_{|n|, |q|< N} \alpha_{n,q}\mathbf{1}_{|n|<|q|}\,,
\]
so that 
\[
    |\gamma - \gamma_N| \leqslant 2\sum_{|q|\geqslant 1}\sum_{|n|\geqslant N} \alpha_{n,q}  \mathbf{1}_{|n|<|q|} +  2\sum_{|n|\geqslant 1}\sum_{|q|\geqslant N} \alpha_{n,q}  \mathbf{1}_{|n|<|q|}\,. 
\]
Using $|n|<|q|$ we have $\langle n\rangle^{2s} \leqslant 2^s |q|^{2s}$ and $\langle q\rangle ^{2s} \leqslant 2^s |q|^{2s}$ as well as $\langle n+q\rangle^{2s} \leqslant 5^s |q|^{2s}$. Therefore a crude bound on $\beta_{n,q}$ reads
\[
    \alpha_{n,q} \mathbf{1}_{|n|<|q|}\leqslant |q|^2|n|^2\frac{1}{|n|^2}(2^{s+2}+5^s)|q|^{2s}\frac{1}{|n|^{10}|q|^{10}} = \frac{2^{3/2}+3^{1/2}}{|n|^{10}|q|^{7}}, 
\]
where we have used $s=\frac{1}{2}>0$. We have also used $\frac{1}{\log(3+\langle n \rangle ^2)}\leqslant 1$. In the following we use the following bounds: 
\begin{itemize}
    \item $2^{5/2}+5^{1/2}<8$, which is obtained by direct computations (by hand). 
    \item For any $\tau \geqslant 4$ we bound 
    \[
        \sum_{n\in\mathbb{Z}^2 \setminus \{0\}} \frac{1}{|n|^{\tau}} \leqslant \sum_{n\in\mathbb{Z}^2 \setminus \{0\}}\frac{1}{|n|^4} = \sum_{k\geqslant 1} |\{n\in\mathbb{Z}^2, |n|=k\}| \frac{1}{k^4}.
    \]
    We use the very simple fact that 
    \[
        \{n\in\mathbb{Z}^2, |n|=k\} \subset \{n\in\mathbb{Z}^2:\max_{|n_1|,|n_2|}\leqslant k\} \setminus \{0\},
    \]
    so that  $|\{n\in\mathbb{Z}^2, |n|=k\}| \leqslant (2k+1)^2 -1=4k^2+4k \leqslant 8k^2$ 
    %\]
    %If we let $D_k \coloneqq \{(n_1,n_2)\in\mathbb{Z}^2, \max\{|n_1|,|
    %n_2|\} \leqslant k\}$, we observe that 
    %\[\{n\in\mathbb{Z}^2, |n|=k\} \subset D_k \setminus D_{k-1}
    %\] 
    %here holds $|\{n\in\mathbb{Z}^2, |n| \leqslant k\}|\leqslant (2k+1)^2-(2k-1)^2=8k$, so that 
    and therefore
    \[
        \sum_{n\in\mathbb{Z}^2 \setminus \{0\}} \frac{1}{|n|^{\tau}} \leqslant \sum_{k\geqslant 1} \frac{8k^2}{k^4} \leqslant \frac{8\pi^2}{6}. 
    \]
\end{itemize}

We can also write:
\[
        \sum_{|q|\geqslant N} \frac{1}{|q|^K} = \sum_{k\geqslant 0}\sum_{ 2^k(N+1)\leqslant |q| < 2^{k+1}N} \frac{1}{|q|^K} \leqslant 12 N^{-K+2}\sum_{k\geqslant 0} 2^{-(K-2)k}
\]
therefore since $\sum_{k\geqslant 0} 2^{-(K-2)k} \leqslant 2$ we have the crude bounds
\[
    \sum_{\substack{|n|<|q|}{|q|\geqslant N, |n|\geqslant 1}} \alpha_{n,q} \leqslant  8 \times \frac{8\pi^2}{6} \times \frac{24}{N^5}, 
\]
and 
\[
    \sum_{\substack{|n|<|q|}{|q|\geqslant 1, |n|\geqslant N}} \alpha_{n,q} \leqslant 8 \times \frac{24^2}{N^{13}}
\]
so that finally
\[
        \frac{1}{2}|\gamma -\gamma_N| \leqslant \frac{1536}{N^5}\left(\frac{\pi^2}{6} + \frac{3}{N^8}\right) < \frac{1536}{N^5}\left(\frac{10}{6} + \frac{3}{N^8}\right) =: \varepsilon,
\]
where we have used $\pi^2 <10$. An interval for the value of $\varepsilon$ is also computed in \cref{app.code}. 

\newpage 
\appendix 

\section{Python code}\label{app.code}

\begin{lstlisting}[language=Python,breaklines]
import numpy as np
from mpmath import iv #this is the interval arithmetic package
from mpmath import *
sig=4 #value of sigma
s=1 #value of s
N=30 #cutoff in the partial sum
def norm(n,m): #defines the square of the norm
	return iv.mpf(n**2+m**2)

def bkt(n,m): #defines the square of Japanese bracket
	return iv.mpf(1+n**2+m**2)

def ps(n,m,p,q):
	return iv.mpf((n*q-m*p)**2)

def a2(n,m): #defines |a_n|^2
	return iv.mpf(1/((bkt(n,m)**(sig+1))*(iv.log(3+bkt(n,m))**2)))
S=iv.mpf(0)
for n in range(-N+1,N,1):
	print('step '+str(n))
	for m in range(-N+1,N,1):
		for p in range(-N+1,N,1):
			for q in range(-N+1,N,1): 
				scalar=n*q-m*p
				result=0
				if(scalar !=0):
					ps=iv.mpf((n*q-m*p)**2)
					Nn=n**2+m**2
					Nq=p**2+q**2
					Nnq=(n+p)**2 + (m+q)**2
					if((Nn < Nq) and (Nn != Nnq) and (Nnq != Nq) and (Nn < N**2) and (Nq < N**2)):
						result=a2(n,m)*a2(p,q)*ps*(1/norm(n,m)-1/norm(p,q))*((bkt(n+p,m+q)**s)*(1/norm(n,m)-1/norm(p,q))+(bkt(n,m)**s)*(1/norm(p,q)-1/norm(n+p,m+q))+(bkt(p,q)**s)*(1/norm(n+p,m+q)-1/norm(n,m)))
				S=S+result
print(S)
print(iv.mpf((1536/(N**5))*(10/6+3/(N**8))))
\end{lstlisting}

\bibliographystyle{alpha}
\bibliography{biblio}
\end{document}